\documentclass{article}

\usepackage{amsmath, amsfonts, amsthm, amssymb}

\usepackage{mathtools}
\usepackage{microtype}

\usepackage{float}
\usepackage{enumerate}
\usepackage{graphicx}

\usepackage{geometry}
\def\R{\mathbb{R}}
\def\F{\mathbb{F}}

\newtheorem{theorem}{Theorem}[section]
\newtheorem{lemma}[theorem]{Lemma}
\newtheorem{proposition}[theorem]{Proposition}
\newtheorem{corollary}[theorem]{Corollary}
\newtheorem*{thm:intro}{Theorem \ref{thm:realization}}

\theoremstyle{definition}
\newtheorem{problem}[theorem]{Problem}
\newtheorem{definition}[theorem]{Definition}
\newtheorem{remark}[theorem]{Remark}
\newtheorem{example}[theorem]{Example}

\def\Int{\operatorname{Int}}
\def\es{\operatorname{S}}
\def\de{\operatorname{D}}
\def\reeb#1{\mathcal{R}(#1)}
\def\ind{\operatorname{ind}}
\def\degin{\deg_{in}}
\def\degout{\deg_{out}}
\def\corank{\operatorname{corank}}

\def\IP{\operatorname{IP}}
\def\DP{\operatorname{DP}}
\def\cl{\operatorname{cl}}

\usepackage{tikz}
\usetikzlibrary{matrix,positioning,calc}
\usetikzlibrary{decorations.pathmorphing}
\usetikzlibrary{decorations.pathreplacing}

\tikzset{snake it/.style={-stealth,
		decoration={snake, 
			amplitude = .4mm,
			segment length = 2mm,
			post length=0.9mm},decorate}}

\makeatletter
\def\@addpunct#1{%
	\relax\ifhmode
	\ifnum\spacefactor>\@m \else#1\fi
	\fi}
\newcommand{\keywordsname}{$2010$ Mathematics Subject Classification}
\def\@setkeywords{%
	{\itshape \keywordsname.}\enspace \@keywords\@addpunct.}
\def\keywords#1{\def\@keywords{#1}}
\let\@keywords=\@empty
\g@addto@macro{\maketitle}{\begingroup%
	\let\@makefnmark\relax  \let\@thefnmark\relax%
	\ifx\@keywords\@mpty\else\@footnotetext{\@setkeywords}\fi%
	\endgroup}
\makeatletter

\keywords{05C76, 57M15, 05C38.  \indent\indent The author was supported by the Polish Research Grants: NCN UMO-2015/19/B/ST1/01458 and \\
	\indent\indent NCN Sheng 1 UMO-2018/30/Q/ST1/00228. \\
	\indent\indent The final publication is available at link.springer.com: https://doi.org/10.1007/s00454-020-00260-6}

\newcommand{\address}{{ \bigskip

\footnotesize
	{\noindent\textsc{\L{}ukasz Patryk Michalak}\\
		Adam Mickiewicz University, Pozna\'n\\
		Faculty of Mathematics and Computer Science\\
		ul.~Uniwersytetu Pozna\'nskiego 4, 61-614 Pozna\'n, Poland} \\
		\textit{E-mail address:} \texttt{lukasz.michalak@amu.edu.pl}
		
}}

\date{}

\title{Combinatorial modifications of Reeb graphs and~the realization problem}

\author{\L{}ukasz Patryk Michalak}

\begin{document}

	\maketitle
	\begin{abstract}
		We prove that, up to homeomorphism, any graph subject to natural necessary conditions on orientation and the cycle rank can be realized as the Reeb graph of a Morse function on a given closed manifold $M$. Along the way, we show that the Reeb number $\mathcal{R}(M)$, i.e. the maximum cycle rank among all Reeb graphs of functions on $M$, is equal to the corank of fundamental group $\pi_1(M)$, thus extending a previous result of Gelbukh to the non-orientable case.
	\end{abstract}

\section{Introduction} 		     

Let $M$ be a closed manifold and $f\colon M\rightarrow \R$ be a~smooth function with finitely many critical points. The Reeb graph of $f$, denoted by $\reeb{f}$, is obtained by contracting the connected components of level sets of the function. It was introduced by Reeb \cite{Reeb} in 1946 and now it plays a~fundamental role in computational topology for shape analysis (see~\cite{Biasotti}). Reeb graphs can also be used for classification of functions, for example they classify up to conjugation simple Morse--Bott functions on closed orientable surfaces (see \cite{Morse-Bott}).

The starting point for this paper is the following natural problem.

\begin{problem}\label{main_problem} For a given manifold $M$, which graph $\Gamma$ can be realized as the Reeb graph of a function $f\colon M \rightarrow \R$ with finitely many critical points?
\end{problem}

The author gave an answer to the question in Problem \ref{main_problem} for surfaces \cite[Theorem 5.4 and 5.6]{Michalak}.
In order to state it we introduce the following notation. By the cycle rank of a graph $\Gamma$ we mean its first Betti number $\beta_1(\Gamma)$. Define the Reeb number $\reeb{M}$ of a manifold $M$ to be the maximum cycle rank among all Reeb graphs of functions on $M$ with finitely many critical points. Recall that each Reeb graph admits the so-called good orientation (see Definition~\ref{definition:good_orientation}, cf.~\cite{Sharko},~\cite{Saeki}). 	
Now, it turns out that, except the complete graph on two vertices, each graph $\Gamma$ with good orientation can be realized up to isomorphism of oriented graphs as the Reeb graph of a function with finitely many critical points on a~given closed surface $\Sigma$, provided that $\beta_1(\Gamma) \leq \reeb{\Sigma}$.

In general, a significant amount of work on Reeb graphs is concerned with functions on surfaces (see \cite{Edelsbrunner}, \cite{Fabio-Landi}, \cite{KMS}, \cite{Kudryavtseva}, \cite{Morse-Bott}, \cite{Saeki}, \cite{Michalak}). Although in a recent paper Gelbukh \cite{Gelbukh:Reeb_graph} described all possible cycle ranks of Reeb graphs of Morse functions on a~closed orientable manifold of an arbitrary dimension $n\geq 2$. The question of realizability of cycle ranks in Reeb graphs is a special case of the problem. 
In this paper we settle Problem \ref{main_problem} as follows.

\begin{thm:intro}
	Let $M$ be a closed, connected $n$-dimensional manifold, $n\geq 2$, and $\Gamma$ be a finite oriented graph. There exists a Morse function $f\colon M \to \R$ such that $\reeb{f}$ is orientation-preserving homeomorphic to~$\Gamma$ if and only if $\Gamma$ has a good orientation and $\beta_1(\Gamma) \leq \reeb{M}$. Moreover, if $M$ is not an orientable surface and the maximum degree of a vertex in $\Gamma$ is not greater than $3$, then $f$ can be taken to be simple. 
\end{thm:intro}

The realization in Theorem \ref{thm:realization} is up to homeomorphism of graphs, not the combinatorial isomorphism, so we lose information about the vertices of degree $2$. The first step in the proof is to reduce the problem to realizability of graphs with vertices of degrees $1$ and $3$ by simple Morse functions, i.e. Morse functions which have critical points on distinct levels. Then the construction begins with the existence of a particular Reeb graph with a~given cycle rank, called the initial graph, which is provided by the proof of Theorem \ref{thm:equivalent_conditions}. From this graph we can obtain each homeomorphism type of graphs using a~finite number of combinatorial modifications of Reeb graphs realized by change of simple Morse functions. These modifications come from handle and Morse theory. We prove in Proposition \ref{proposition:canonical_form} that each Reeb graph can be transformed to a canonical form using a finite number of such modifications. In the case of orientable surfaces similar operations have been used by Kudryavtseva \cite[Theorem~1]{Kudryavtseva} and Fabio--Landi \cite[Lemma 2.6]{Fabio-Landi} to prove an~analogous reduction to a~canonical form.

The aforementioned Theorem \ref{thm:equivalent_conditions} is crucial. It states the equivalence of the following three conditions for a closed manifold $M$:
\begin{itemize}
	\item the existence of epimorphism $\pi_1(M)\to\F_r$ onto the free group of rank $r$,
	\item the existence of $r$ disjoint submanifolds $N_1,\ldots, N_r \subset M$ of codimension one with product neighbourhoods, removal of whose does not disconnect~$M$,
	\item the existence of a Morse function on $M$ (simple, if $M$ is not an orientable surface) whose Reeb graph has cycle rank equal to $r$.
\end{itemize}
We conclude that the Reeb number $\reeb{M}$ is equal to corank of fundamental group of $M$ and that any number not greater than $\reeb{M}$ occurs as the cycle rank of the Reeb graph of a Morse function. This theorem is an extension of Cornea \cite[Theorem 1]{Cornea} and Jaco \cite[Theorem 2.1]{Jaco}, where the first two conditions are considered (we note that Jaco works in the category of combinatorial manifolds), and it is a generalization of Gelbukh \cite[Theorem 13]{Gelbukh:Reeb_graph} to non-orientable manifolds.

The paper is organized as follows. In Section \ref{section:basic_notions} we introduce basic notions and properties of Reeb graphs. Section \ref{section:index_and_degree} establishes the relation between the index of a critical point of a simple Morse function and the degree of the corresponding vertex in the Reeb graph. We conclude that the Reeb graph of a~self-indexing Morse function is a tree. In Section \ref{section:combinatorial_mod} we introduce the combinatorial modifications and perform the reduction of the Reeb graph of a simple Morse function to the canonical form. As a consequence, any number between $0$~and $\reeb{M}$ can be realized as the cycle rank of the Reeb graph of a~Morse function on a given manifold $M$. In Section \ref{section:reeb_number_and_corank} we prove Theorem \ref{thm:equivalent_conditions}. Finally, in Section \ref{section:realization} we prove Theorem \ref{thm:realization}.


\section{Basic notions}\label{section:basic_notions}

Throughout the paper we assume that all manifolds are compact, smooth, connected of dimension $n \geq 2$ and that all graphs are finite and connected.

A \textbf{smooth triad} is a triple $(W,W_-,W_+)$, where $W$ is a manifold and its boundary $\partial W = W_- \sqcup W_+$ is the disjoint union of $W_-$ and $W_+$ (possibly $W_\pm = \varnothing$). A~function on a smooth triad $(W,W_-,W_+)$ is a smooth function $f\colon W \to [a,b]$ such that $f^{-1}(a) = W_-$, $f^{-1}(b)=W_+$ and all critical points of $f$ are contained in $\Int W$. 

\begin{definition}
	Let $f\colon W \to \R$ be a function with finitely many critical points on a smooth triad $(W,W_-,W_+)$. We define the \textbf{Reeb relation} $\sim_{\mathcal{R}}$ on $W$: $x\sim_{\mathcal{R}}y$ if and only if they are in the same connected component of a level set of~$f$. The quotient space $W/\!\sim_{\mathcal{R}}$ is denoted by $\reeb{f}$ and called the \textbf{Reeb graph} of the function $f$.
\end{definition}

The Reeb graph of the function $f$ as above is homeomorphic to a finite graph, i.e. to a one-dimensional finite CW-complex (see \cite{Reeb}, \cite{Sharko}). The vertices of~$\reeb{f}$ corresponds to the components of $W_\pm$ and to the components of level sets of $f$ containing critical points. By the quotient topology, $f$ induces the continuous function $\overline{f}\colon \reeb{f} \to \R$ such that $f = \overline{f}\circ q$, where $q\colon W \to \reeb{f}$ is the quotient map.

The \textbf{degree} $\deg(v)$ of a vertex $v$ in a graph $\Gamma$ is the number of edges incident to~$v$. If $\Gamma$ is an oriented graph (i.e. directed --- each edge has a chosen direction), then the \textbf{indegree} (\textbf{outdegree}) of $v$ is the number of edges incoming (outgoing) to $v$ and it is denoted by $\degin(v)$ ($\degout(v)$). Thus $\deg(v) = \degin(v)+\degout(v)$.

\begin{figure}[h]
	\centering
	\begin{tikzpicture}[scale=0.9]

		\draw[rotate=90, scale = 0.8] (-1,0) to[bend left] (1,0);
		\draw[rotate=90, scale = 0.8] (-1.2,.1) to[bend right] (1.2,.1);
		\draw[rotate=90, scale = 0.5] (0,0) ellipse (100pt and 50pt);
		
		\draw [->] (1.5,0) -- (2.5,0);
		\draw (2,0.25) node { $q$};

		\draw[dashed] (-1,1.75) -- (8,1.75);
		\draw[dashed] (-1,-1.75) -- (8,-1.75);
		\draw[dashed] (-1,0.8) -- (8,0.8);
		\draw[dashed] (-1,-0.8) -- (8,-0.8);
		
		\draw [->] (4.8,0) -- (5.8,0);
		\draw (5.3,0.28) node { $\overline{f}$};
		
		\draw (3.5,0.8) -- (3.5,1.75);
		\draw (3.5,-0.8) -- (3.5,-1.75);

		\draw (3.5,0.8) to[out=225,in=135] (3.5,-0.8);
		\draw (3.5,0.8) to[out=315,in=45] (3.5,-0.8);
		

		\draw (7,2) -- (7,-2);
		

		\filldraw (0,0.8) circle (1.5pt);
		\filldraw (0,-0.8) circle (1.5pt);
		\filldraw (0,1.75) circle (1.5pt);
		\filldraw (0,-1.75) circle (1.5pt);
		
		\filldraw (3.5,0.8) circle (1.5pt) ;
		\filldraw (3.5,-0.8) circle (1.5pt) ;
		\filldraw (3.5,1.75) circle (1.5pt) ;
		\filldraw (3.5,-1.75) circle (1.5pt);
		
		\filldraw (7,0.8) circle (1.5pt) ;
		\filldraw (7,-0.8) circle (1.5pt) ;
		\filldraw (7,1.75) circle (1.5pt) ;
		\filldraw (7,-1.75) circle (1.5pt);

		\draw (0,-2.5) node { $\es^1\times\es^1$};
		\draw (3.5,-2.5) node { $\reeb{f}$};
		\draw (7,-2.5) node { $\R$};

		\draw[->] (0.1,-2.75) to[out=345,in=195] (6.8,-2.75);
		\draw (3.5,-3.53) node { $f$};
		
		\draw [->] (1,-2.5) -- (2.7,-2.5);
		\draw (1.85,-2.25) node { $q$};
		
		\draw [->] (4.3,-2.5) -- (6.3,-2.5);
		\draw (5.1,-2.22) node { $\overline{f}$};

		%
		%
	\end{tikzpicture}
	\caption{Example of the Reeb graph of height function $f\colon \es^1\times\es^1 \to \R$ on two-dimensional torus. Four critical points of $f$ have different values and correspond to four vertices in $\reeb{f}$.}\label{figure:Reeb_graph_example}
\end{figure}

\begin{definition}\label{definition:good_orientation}
	A \textbf{good orientation} of a graph $\Gamma$ is the orientation induced by a continuous function $\Gamma \to \R$ that has extrema only in the vertices of degree $1$ and which is strictly monotonic on the edges.
\end{definition}

It is easy to see that the function $\overline{f}$ on $\reeb{f}$ induces a good orientation (cf.~Figure \ref{figure:Reeb_graph_example}). This important concept was studied comprehensively by Sharko \cite{Sharko}, while the name we use comes from Masumoto--Saeki \cite{Saeki}. Sharko used an equivalent definition: an oriented graph has a good orientation if and only if the following conditions are satisfied:
\begin{itemize}
	\item it has at least two vertices of degree $1$, one with incoming and one with outgoing edge,
	\item any vertex of a higher degree has both incoming and outgoing edge,
	\item it does not have oriented cycles.
\end{itemize}
A basic observation is that a graph with good orientation does not have loops, i.e. edges that connect a vertex to itself. Sharko provided an example of a graph which cannot admit a good orientation (see Figure \ref{figure:good_orientation_counterexample} (a)). Any attempt to orient this graph causes failure of one of the above conditions. However, the graph presented in Figure \ref{figure:good_orientation_counterexample} (b) can be oriented in a good way.
\begin{figure}[h]
	\centering
	\begin{tikzpicture}[scale=0.75]
		
		\draw (0,2) -- (0,1);
		\draw (0,1) -- (0,0);
		\draw (0,1) -- (1,1);
		\draw (2,0) -- (2,2);
		
		\draw (2,1) circle (1cm);
		
		\filldraw (0,0) circle (1.5pt);
		\filldraw (0,1) circle (1.5pt) ;
		\filldraw (0,2) circle (1.5pt) ;
		\filldraw (1,1) circle (1.5pt);
		\filldraw (2,2) circle (1.5pt);
		\filldraw (2,0) circle (1.5pt) ;
		
		\draw (1.5,-0.5) node {(a)};
		
		
		\draw (6,2) -- (6,1);
		\draw (6,1) -- (6,0);
		\draw (6,1) -- (7,1);
		\draw (8,0) -- (8,2);
		\draw (8,0) -- (8.5,1);
		
		\draw (8,1) circle (1cm);
		
		\filldraw (6,0) circle (1.5pt);
		\filldraw (6,1) circle (1.5pt) ;
		\filldraw (6,2) circle (1.5pt) ;
		\filldraw (7,1) circle (1.5pt); 
		\filldraw (8,2) circle (1.5pt); 
		\filldraw (8,0) circle (1.5pt); 
		\filldraw (8.5,1) circle (1.5pt);
		
		\draw (7.5,-0.5) node {(b)};
		
	\end{tikzpicture}
	\caption{(a) Graph not admitting a good orientation \cite[Figure 1]{Sharko} and (b) graph admitting~it.}
	\label{figure:good_orientation_counterexample}
\end{figure}

Let $\Gamma$ be a graph with good orientation induced by $g\colon \Gamma \to \R$. A path $\tau \colon [0,1] \rightarrow \Gamma$ is \textbf{increasing}, if $g(\tau(t)) < g(\tau(t'))$ for $0 \leq t < t' \leq 1$. Similarly we define a \textbf{decreasing} path. For two vertices $v$ and $w$ in $\Gamma$ we say that $v$ is \textbf{below} $w$ and $w$ is \textbf{above} $v$ if there is an increasing path from $v$ to $w$. It is clear that these definitions do not depend on the choice of the function $g$. According to this notation, orientations of graphs presented in figures in this paper are from the bottom to the top, as we see in Figure \ref{figure:Reeb_graph_example}.

Let $\Gamma = \Gamma(V,E)$ be a graph, where $V$ and $E$ are the sets of vertices and edges of $\Gamma$, respectively. The \textbf{cycle rank} of $\Gamma$ is defined by its first Betti number $\beta_1(\Gamma)$. Clearly, $\beta_1(\Gamma) = |E|-|V|+1$ and the fundamental~group~of~$\Gamma$~is $\pi_1(\Gamma) \cong \F_{\beta_1(\Gamma)}$, where $\F_r$ is the free group of rank $r\geq 0$ ($\F_0$ is the trivial~group).

\begin{remark}
	The cycle rank of a graph is also called the \emph{number of loops} \mbox{(cf. \cite{Edelsbrunner,Gelbukh:Reeb_graph}).}
\end{remark}

\begin{definition}
	The \textbf{Reeb number} $\reeb{M}$ of a closed manifold $M$ is the maximum cycle rank among all Reeb graphs of smooth functions on $M$ with finitely many critical points.
\end{definition}

From \cite{KMS} we know that $\beta_1(\reeb{f}) \leq \beta_1(M)$, so $\reeb{M}$ is well-defined and $\reeb{M}\leq \beta_1(M)$. If $\Sigma$ is a closed surface of the Euler characteristic $\chi(\Sigma) = 2-k$, then $\reeb{\Sigma} = \left\lfloor \frac{k}{2} \right\rfloor$, where $\left\lfloor x\right\rfloor$ is the floor of $x$ (see \cite{Gelbukh:Reeb_graph}, \cite[Corollary 3.8]{Michalak}).

\begin{remark}\label{remark:reeb_nubmer_continuous}
	Note that the notions of Reeb graph and Reeb number can be considered for larger classes of spaces and maps. Gelbukh \cite{Gelbukh:filomat} investigated the class of continuous functions $f\colon X \to \R$ on a connected and locally path-connected topological space $X$ with the additional assumption that the quotient space $\reeb{f}$ obtained by using the Reeb relation is a finite topological graph. She easily showed that $\beta_1(\reeb{f})$ is bounded in such a case by the corank of $\pi_1(X)$ (see Definition \ref{definition:corank}). Considering the above class of continuous functions in the definition of the Reeb number of a closed manifold  provides the same quantity (cf.~Corollary~\ref{corollary:corank=reeb_number}). However, we do not know and there are not provided any necessary and sufficient conditions on a continuous function $f$ under which $\reeb{f}$ is a finite graph. What is known is the case of smooth functions with isolated critical points on a compact manifold (thus with a finite number of them by compactness) which is considered in this paper. It is worth pointing out that Gelbukh also provided examples of spaces and maps for which the inequality $\beta_1(\reeb{f}) \leq \corank(\pi_1(X))$ does not hold (e.g. for the Warsaw circle).

\end{remark}


\section{Index and degree correspondence} \label{section:index_and_degree}

In this section we discuss how some properties of Morse functions affect the Reeb graph and its cycle rank. First of all, we present the correspondence between the index of critical point of a simple Morse function and the degree of the corresponding vertex in the Reeb graph (Proposition \ref{proposition:correspondece_degree_index}). This is one of the basic ingredients needed in the realization theorem (Theorem \ref{thm:realization}) and for introducing combinatorial modifications of Reeb graphs. We also show that one can always take a function whose Reeb graph is a tree (Proposition \ref{proposition:ordered_has_tree}) what will be used in the realization of a graph in the initial form as the Reeb graph (Theorem \ref{thm:equivalent_conditions} and Proposition \ref{proposition:initial_graph}).

Recall that a smooth function $f\colon W \rightarrow \R$ on a triad $(W,W_-,W_+)$ is called a \textbf{Morse function} if all its critical points are non-degenerate. The index of a non-degenerate critical point $p$ is denoted by $\ind(p)$. We say that $f$ is:
\begin{itemize}
	\item \textbf{simple} if on every critical level there is exactly one critical point,
	\item \textbf{self-indexing} if there exist values $c_0 < c_1 < \ldots < c_n$ (where $n=\dim W$) such that for any critical point $p$ if $\ind(p) = i$, then $f(p)=c_i$,
	\item \textbf{ordered} if for any two critical points $p$ and $p'$ if $\ind(p) < \ind(p')$, then $f(p) < f(p')$.
\end{itemize}

In other words, ordered Morse functions have the property that critical points of a smaller index are below the critical points of a larger index. Self-indexing Morse functions are obviously ordered. In fact, they can be defined as ordered Morse functions for which all critical points of the same index have the same value. By \cite[ Lemma 2.8 and Theorem 4.8]{Milnor} for any Morse function on a~triad $(W,W_-,W_+)$ there exist a self-indexing Morse function and a~simple Morse function, both with the same critical points each with the same index.

By \cite[Lemma 3.5]{Michalak} the Reeb number of a closed manifold $M$ can be attained by simple Morse functions, so
\[
\reeb{M} =\max \left\{\, \beta_1(\reeb{f})\, |\, f\colon M \rightarrow \R \text{ is a simple Morse function}\, \right\}\!.
\]
On the other hand, Reeb graphs of self-indexing Morse functions are always trees. In general, ordered Morse functions have also Reeb graphs with no cycles if a manifold is of dimension at least three (see Proposition \ref{proposition:ordered_has_tree}).

For a manifold $W$, a function $f\colon W\to \R$, $c\in \R$ and any interval $I\subset \R$ we use the following notation:
$$W_c := f^{-1}(c),\qquad W^c := f^{-1}((-\infty,c]),\qquad W^I := f^{-1}(I).$$

The following proposition comes from original paper of Reeb \cite{Reeb}.

\begin{proposition}[{\cite[Th\'{e}or\`{e}me 3]{Reeb}}] \label{proposition:correspondece_degree_index}
	Let $f\colon W \rightarrow \R$ be a simple Morse function on an $n$-dimensional smooth triad $(W,W_-,W_+)$, $n \geq 3$. Let $p$ be a~critical point of $f$ and $v:= q(p)$ be the vertex in $\reeb{f}$ which corresponds to~$p$. Then
	$$
	\deg (v) =  \begin{cases}
		1	        &\ \ \ \ \text{if $\ind(p) = 0 \text{ or } n$}, \\
		2 \text{ or } 3	        &\ \ \ \ \text{if $\ind(p) = 1 \text{ or } n-1$},\\
		2		  &\ \ \ \ \text{in other cases}.
	\end{cases}$$
	$$\ind(p) =  \begin{cases}
		0 \text{ or } n	      				  &\ \ \ \ \text{if $\deg(v) = 1$},\\
		1         				  &\ \ \ \ \text{if $\deg(v) = 3$ and $\degin(v) = 2$}, \\
		n-1	      				  &\ \ \ \ \text{if $\deg(v) = 3$ and $\degout(v) = 2$},\\
		1  \text{ or }\ldots  \text{ or } n-1			  &\ \ \ \ \text{if $\deg(v) = 2$}.
	\end{cases}$$
\end{proposition}
\begin{proof}
	The proposition follows easily from the fact, that by Morse theory the manifold $W^{c+\varepsilon}$ is homeomorphic to the manifold $W^{c-\varepsilon}$ with $k$-handle attached by an embedding  $\varphi \colon \es^{k-1} \times \de^{n-k} \rightarrow W_{c-\varepsilon}$, where $c=f(p)$, $k=\ind(p)$ and $\varepsilon >0$ is sufficiently small. 
\end{proof}
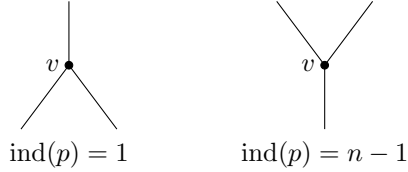
\begin{figure}[h]
	\centering
	\begin{tikzpicture}[scale=0.9]

		\filldraw (0,0) circle (1.7pt)   node[align=left, right] {} node[align=right, left] {$v$};
		\filldraw (4,0) circle (1.7pt)   node[align=left, right] {} node[align=right, left] {$v$};
		
		\draw (0,0) -- (0,1);
		\draw (0,0) -- (-0.75,-1);
		\draw (0,0) -- (0.75,-1);

		\draw (4,0) -- (4,-1);
		\draw (4,0) -- (3.25,1);
		\draw (4,0) -- (4.75,1);
		
		\draw (0,-1.4) node{$\ind(p)=1$};
		\draw (4,-1.4) node{$\ind(p)=n-1$};

	\end{tikzpicture}
	\caption{Possible neighbourhoods of a vertex of degree $3$ in the Reeb graph of a~simple Morse function on a manifold of dimension at least $3$. We use our convention that the orientation is from the bottom to the top.}\label{rys:deg_3}
\end{figure}

For simplicity, we define the \textbf{index} of $v$ to be the index of $p$. We also extend this definition for arbitrary graphs with good orientation --- the index of a~vertex $v$ of degree $3$ is $1$ if $\degin(v)=2$ and $n-1$ if $\degout(v)=2$.

By the correspondence in the above proposition if $\reeb{f}$ contains a cycle, then some 
vertex of degree 3 of index $n-1$ is below a vertex of degree $3$ of index~$1$.

\begin{proposition}\label{proposition:ordered_has_tree}
	Let $f\colon W\rightarrow \R$ be an ordered Morse function on a smooth triad $(W,W_-,W_+)$ of dimension $n\geq 3$. Then $\reeb{f}$ is a tree. In particular, the Reeb graph of a~self-indexing Morse function is a tree (even for $n=2$).
\end{proposition}

\begin{proof}
	Every critical level of $f$ contains critical points of the same index. Therefore by \cite[Lemma 3.5 and Remark 3.6]{Michalak} performing a small perturbation of $f$ we can obtain a simple and still ordered Morse function $g\colon W \rightarrow \R$ with the same critical points and of the same index as $f$ such that $\beta_1(\reeb{f}) \leq \beta_1(\reeb{g})$.
	
	Suppose that the Reeb graph $\reeb{g}$ has a cycle. From the above proposition there are two vertices of degree $3$ and of index $1$ and $n-1$ which are the highest and the lowest vertices in this cycle, respectively. This is a contradiction, since $g$ is ordered. Hence $\reeb{g}$ is a tree, so $\reeb{f}$ also.
	
	For the case of a self-indexing Morse function on a surface see a comment below \cite[Lemma 3.2]{Michalak}. 
\end{proof}

\begin{lemma}\label{lemma:ordered_function_is_tree_with_connected_preimage}
	For any ordered Morse function  $f\colon W \to \R$ on a smooth triad $(W,W_-,W_+)$, $\dim W \geq 3$, there exist a regular value $c$ such that $W_c$ is connected and that $\partial\left(W^{(-\infty,c)}\right) = W_-$ and $\partial\left( W^{(c,+\infty)}\right) = W_+$. 
\end{lemma}

\begin{proof}
	We may assume that $f$ is simple by changing it on arbitrary small neighbourhoods of critical points.
	Suppose that $f$ has critical points of indices $1$ and $n-1$. Then by Proposition \ref{proposition:correspondece_degree_index} a subgraph of $\reeb{f}$ between the highest vertex of index $1$ and the lowest vertex of index $n-1$ is homeomorphic to the interval and $c$ can be taken from levels corresponding to this subgraph. If there is no vertex of index $1$ (for $n-1$ we proceed analogously), then $W_-=\varnothing$ and $f$ has one minimum or $W_-\neq \varnothing$ is connected and $f$ does not have any critical point being a minimum. In both the cases $\reeb{f}$ has a unique vertex $v$ with indegree $0$. Thus we may take any regular value through which the edge incident to $v$ passes. 
\end{proof}

\begin{lemma}\label{lemma:number_of_cycles_Delta_3}
	Let $f\colon M \rightarrow \R$ be a simple Morse function on an $n$-dimensional manifold $M$, $n \geq 3$, $k_i$ be the number of critical points of index $i$ and let $\Delta_3$ be the number of vertices of degree 3 in $\reeb{f}$. Then $$\beta_1(\reeb{f}) = -\frac{k_0+k_n}{2} + \frac{\Delta_3}{2} + 1.$$
	Furthermore, if we denote by $\Delta_3^{\operatorname{in}}$ (by  $\Delta_3^{\operatorname{out}}$) the number of vertices in $\reeb{f}$ with indegree $2$ (outdegree $2$), then $$\Delta_3^{\operatorname{in}} - k_0 + 1 = \beta_1(\reeb{f}) = \Delta_3^{\operatorname{out}} - k_n +1$$
\end{lemma}

\begin{proof} It is an easy computation since $|V| = \sum_{i} k_i$, $2|E| = \sum_{v\in V} \deg(v)$ and $\beta_1(\reeb{f}) = |E|-|V|+1$. The second part is just a careful investigation of graphs with vertices of degrees $1$, $2$ and $3$. 
\end{proof}


\section{Combinatorial modifications of Reeb graphs}\label{section:combinatorial_mod}

\begin{lemma} \label{lemat:ruchy}
	Let $f\colon W \rightarrow \R$ be a simple Morse function with exactly two critical points $p$ and $p'$ on a smooth triad $(W,W_-,W_+)$, where $n = \dim W \geq 3$. Let $f(p) > f(p')$ and assume that $\reeb{f}$ is isomorphic to the graph on the left side of the case $(i)$ in Figure \ref{figure:modifications}. If $\ind(p) \leq \ind(p')$, then there exists a simple Morse function $g\colon W \rightarrow \R$ with the same critical points and of the same index as $f$, such that $g(p) < g(p')$ and $\reeb{g}$ is isomorphic to the graph on the right side of  the case $(i)$ in Figure \ref{figure:modifications} (in the cases (4) and (5) we require the order of the vertices corresponding to the components of $W_\pm = V^\pm_1 \sqcup V^\pm_2 \sqcup V^\pm_3$ determined by a permutation $\sigma\colon \{1,2,3\} \to \{1,2,3\}$).
\end{lemma}

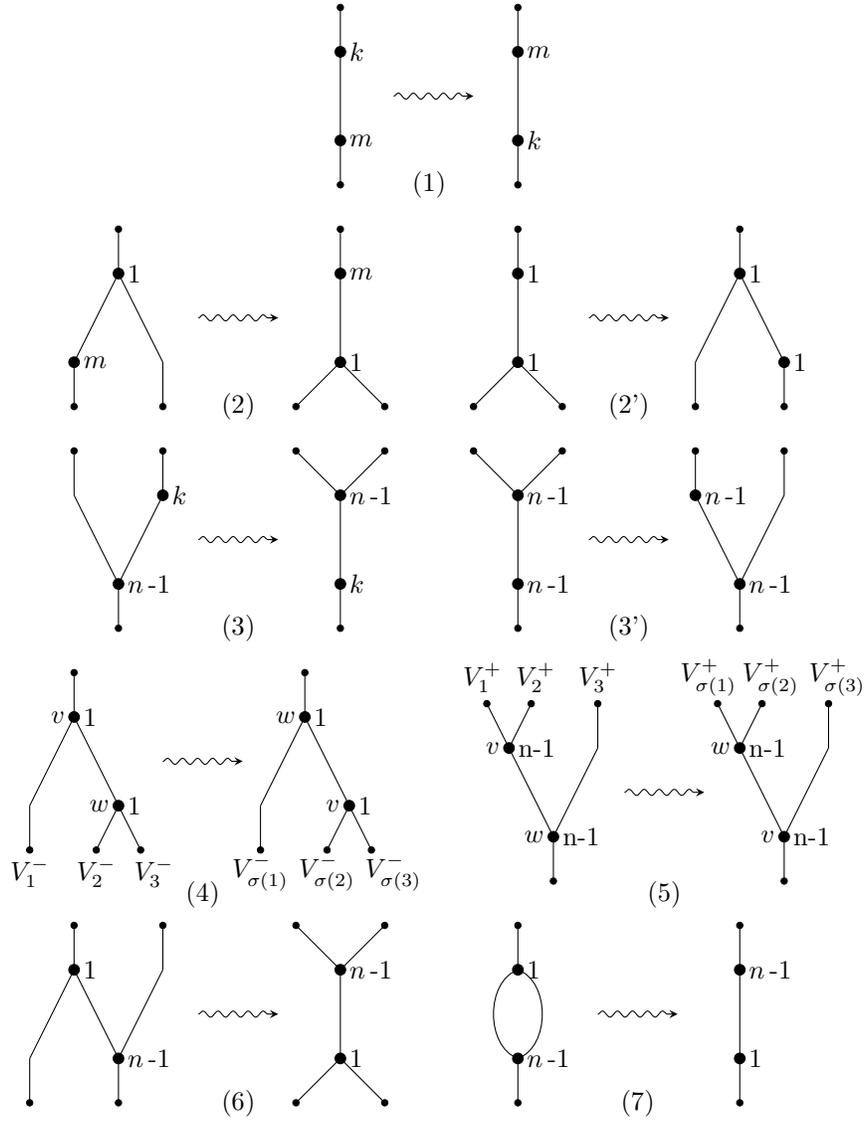
\begin{figure}[h]
	\centering
	
	\begin{tikzpicture}[scale=1]

		
		\draw (-1,0.5) -- (-1,2.5);
		\filldraw (-1,2) circle (1.7pt)   node[align=left, right] {$k$};
		\filldraw (-1,1) circle (1.7pt)   node[align=left, right] {$m$};
		
		\draw (1,0.5) -- (1,2.5);
		\filldraw (1,2) circle (1.7pt)   node[align=left, right] {$m$};
		\filldraw (1,1) circle (1.7pt)   node[align=left, right] {$k$};
		
		\draw [->, looseness=2, snake it ] (-0.4,1.5) -- (0.5,1.5);
		\draw (0,0.5) node{(1)};
		
		\filldraw (-1,0.5) circle (1pt);
		\filldraw (-1,2.5) circle (1pt);
		\filldraw (1,0.5) circle (1pt);
		\filldraw (1,2.5) circle (1pt);

		
		\draw (-3.5,0) -- (-3.5,-0.5);
		\draw (-3,-1.5) -- (-3.5,-0.5);
		\draw (-4,-1.5) -- (-3.5,-0.5);
		\draw (-4,-1.5) -- (-4,-2);
		\draw (-3,-1.5) -- (-3,-2);
		\filldraw (-3.5,-0.5) circle (1.7pt)   node[align=left, right] {1};
		\filldraw (-4,-1.5) circle (1.7pt)   node[align=left, right] {$m$};
		
		\draw (-1,0) -- (-1,-1.5);
		\draw (-0.5,-2) -- (-1,-1.5);
		\draw (-1.5,-2) -- (-1,-1.5);
		\filldraw (-1,-0.5) circle (1.7pt)   node[align=left, right] {$m$};
		\filldraw (-1,-1.5) circle (1.7pt)   node[align=left, right] {1};
		
		\draw [->, looseness=2, snake it ] (-2.6,-1) -- (-1.7,-1);
		\draw (-2.15,-2) node{(2)};

		\filldraw (-3.5,0) circle (1pt);
		\filldraw (-3,-2) circle (1pt);
		\filldraw (-4,-2) circle (1pt);
		
		\filldraw (-1,0) circle (1pt);
		\filldraw (-0.5,-2) circle (1pt);
		\filldraw (-1.5,-2) circle (1pt);

		
		\draw (3.5,0) -- (3.5,-0.5);
		\draw (3,-1.5) -- (3.5,-0.5);
		\draw (4,-1.5) -- (3.5,-0.5);
		\draw (4,-1.5) -- (4,-2);
		\draw (3,-1.5) -- (3,-2);
		\filldraw (3.5,-0.5) circle (1.7pt)   node[align=left, right] {1};
		\filldraw (4,-1.5) circle (1.7pt)   node[align=left, right] {1};
		
		\draw (1,0) -- (1,-1.5);
		\draw (0.5,-2) -- (1,-1.5);
		\draw (1.5,-2) -- (1,-1.5);
		\filldraw (1,-0.5) circle (1.7pt)   node[align=left, right] {1};
		\filldraw (1,-1.5) circle (1.7pt)   node[align=left, right] {1};
		
		\draw [->, looseness=2, snake it ] (1.8,-1) -- (2.7,-1);
		\draw (2.25,-2) node{(2')};

		\filldraw (3.5,0) circle (1pt);
		\filldraw (3,-2) circle (1pt);
		\filldraw (4,-2) circle (1pt);
		
		\filldraw (1,0) circle (1pt);
		\filldraw (0.5,-2) circle (1pt);
		\filldraw (1.5,-2) circle (1pt);
		
		
		\draw (-3.5,-4) -- (-3.5,-4.5);
		\draw (-3,-3) -- (-3.5,-4);
		\draw (-4,-3) -- (-3.5,-4);
		\draw (-4,-2.5) -- (-4,-3);
		\draw (-3,-2.5) -- (-3,-3);
		\filldraw (-3,-3) circle (1.7pt)   node[align=left, right] {$k$};
		\filldraw (-3.5,-4) circle (1.7pt)   node[align=left, right] {$n\,$-1};

		\draw (-1,-3) -- (-1,-4.5);		
		\draw (-0.5,-2.5) -- (-1,-3);
		\draw (-1.5,-2.5) -- (-1,-3);
		\filldraw (-1,-3) circle (1.7pt)   node[align=left, right] {$n\,$-1};
		\filldraw (-1,-4) circle (1.7pt)   node[align=left, right] {$k$};
		
		\draw [->, looseness=2, snake it ] (-2.6,-3.5) -- (-1.7,-3.5);
		\draw (-2.15,-4.5) node{(3)};

		\filldraw (-3.5,-4.5) circle (1pt);
		\filldraw (-3,-2.5) circle (1pt);
		\filldraw (-4,-2.5) circle (1pt);
		
		\filldraw (-1,-4.5) circle (1pt);
		\filldraw (-0.5,-2.5) circle (1pt);
		\filldraw (-1.5,-2.5) circle (1pt);
		
		
		\draw (3.5,-4) -- (3.5,-4.5);
		\draw (3,-3) -- (3.5,-4);
		\draw (4,-3) -- (3.5,-4);
		\draw (4,-2.5) -- (4,-3);
		\draw (3,-2.5) -- (3,-3);
		\filldraw (3,-3) circle (1.7pt)   node[align=left, right] {$n\,$-1};
		\filldraw (3.5,-4) circle (1.7pt)   node[align=left, right] {$n\,$-1};
		
		\draw (1,-3) -- (1,-4.5);
		
		\draw (0.5,-2.5) -- (1,-3);
		\draw (1.5,-2.5) -- (1,-3);
		\filldraw (1,-3) circle (1.7pt)   node[align=left, right] {$n\,$-1};
		\filldraw (1,-4) circle (1.7pt)   node[align=left, right] {$n\,$-1};
		
		\draw [->, looseness=2, snake it ] (1.8,-3.5) -- (2.7,-3.5);
		\draw (2.25,-4.5) node{(3')};

		\filldraw (3.5,-4.5) circle (1pt);
		\filldraw (3,-2.5) circle (1pt);
		\filldraw (4,-2.5) circle (1pt);
		
		\filldraw (1,-4.5) circle (1pt);
		\filldraw (0.5,-2.5) circle (1pt);
		\filldraw (1.5,-2.5) circle (1pt);

		
		\draw (-3.5,-6.5) -- (-4,-5.5);
		\draw (-4,-5) -- (-4,-5.5);
		\draw (-3.5,-6.5) -- (-3.25,-7);
		\draw (-3.5,-6.5) -- (-3.75,-7);
		\draw (-4.5,-6.5) -- (-4,-5.5);
		\draw (-4.5,-6.5) -- (-4.5,-7);
		
		\draw (-3.1,-7.25) node{$V^-_3$};
		\draw (-3.75,-7.25) node{$V^-_2$};
		\draw (-4.5,-7.25) node{$V^-_1$};
		
		\filldraw (-4,-5.5) circle (1.7pt)   node[align=left, right] {1} node[align=right, left] {$v$};
		\filldraw (-3.5,-6.5) circle (1.7pt)   node[align=left, right] {1} node[align=right, left]{$w$};

		\draw (-1.4,-5) -- (-1.4,-5.5);
		\draw (-0.9,-6.5) -- (-1.4,-5.5);
		\draw (-0.9,-6.5) -- (-0.65,-7);
		\draw (-0.9,-6.5) -- (-1.15,-7);
		\draw (-1.9,-6.5) -- (-1.4,-5.5);
		\draw (-1.9,-6.5) -- (-1.9,-7);
		
		\draw (-0.4,-7.25) node{$V^-_{\sigma(3)}$};
		\draw (-1.15,-7.25) node{$V^-_{\sigma(2)}$};
		\draw (-1.9,-7.25) node{$V^-_{\sigma(1)}$};
		
		\filldraw (-1.4,-5.5) circle (1.7pt)   node[align=left, right] {1} node[align=right, left] {$w$};
		\filldraw (-0.9,-6.5) circle (1.7pt)   node[align=left, right] {1} node[align=right, left]{$v$};

		\draw [->, looseness=2, snake it ] (-3.0,-6) -- (-2.1,-6);
		\draw (-2.55,-7.5) node{(4)};

		\filldraw (-4,-5) circle (1pt);
		\filldraw (-3.25,-7) circle (1pt);
		\filldraw (-3.75,-7) circle (1pt);
		\filldraw (-4.5,-7) circle (1pt);

		\filldraw (-1.4,-5) circle (1pt);
		\filldraw (-0.65,-7) circle (1pt);
		\filldraw (-1.15,-7) circle (1pt);
		\filldraw (-1.9,-7) circle (1pt);
		
		
		\draw (3.5,-5.85) -- (4,-6.85);
		\draw (4,-7.35) -- (4,-6.85);
		\draw (3.5,-5.85) -- (3.25,-5.35);
		\draw (3.5,-5.85) -- (3.75,-5.35);
		\draw (4.5,-5.85) -- (4,-6.85);
		\draw (4.5,-5.85) -- (4.5,-5.35);
		
		\draw (3.15,-5.05) node{$V^+_{\sigma(1)}$};
		\draw (3.85,-5.05) node{$V^+_{\sigma(2)}$};
		\draw (4.6,-5.05) node{$V^+_{\sigma(3)}$};
		
		\filldraw (4,-6.85) circle (1.7pt)   node[align=left, right] {n-1} node[align=right, left] {$v$};
		\filldraw (3.5,-5.85) circle (1.7pt)   node[align=left, right] {n-1} node[align=right, left]{$w$};

		\draw (1.4,-7.35) -- (1.4,-6.85);
		
		\draw (0.9,-5.85) -- (1.4,-6.85);
		\draw (0.9,-5.85) -- (0.65,-5.35);
		\draw (0.9,-5.85) -- (1.15,-5.35);
		\draw (1.9,-5.85) -- (1.4,-6.85);
		\draw (1.9,-5.85) -- (1.9,-5.35);

		\draw (0.6,-5.05) node{$V^+_1$};
		\draw (1.2,-5.05) node{$V^+_2$};
		\draw (1.9,-5.05) node{$V^+_3$};
		
		\filldraw (1.4,-6.85) circle (1.7pt)   node[align=left, right] {n-1} node[align=right, left] {$w$};
		\filldraw (0.9,-5.85) circle (1.7pt)   node[align=left, right] {n-1} node[align=right, left]{$v$};

		\draw [->, looseness=2, snake it ] (2.2,-6.35) -- (3.1,-6.35);
		\draw (2.65,-7.5) node{(5)};

		\filldraw (4,-7.35) circle (1pt);
		\filldraw (3.25,-5.35) circle (1pt);
		\filldraw (3.75,-5.35) circle (1pt);
		\filldraw (4.5,-5.35) circle (1pt);

		\filldraw (1.4,-7.35) circle (1pt);
		\filldraw (0.65,-5.35) circle (1pt);
		\filldraw (1.15,-5.35) circle (1pt);
		\filldraw (1.9,-5.35) circle (1pt);

		
		\draw (-3.5,-9.35) -- (-3.5,-9.85);
		\draw (-3,-8.35) -- (-3.5,-9.35);
		\draw (-4,-8.35) -- (-3.5,-9.35);
		\draw (-4,-8.35) -- (-4.5,-9.35);
		\draw (-4.5,-9.85) -- (-4.5,-9.35);
		
		\draw (-4,-7.85) -- (-4,-8.35);
		\draw (-3,-7.85) -- (-3,-8.35);
		\filldraw (-4,-8.35) circle (1.7pt)   node[align=left, right] {1};
		\filldraw (-3.5,-9.35) circle (1.7pt)   node[align=left, right] {$n\,$-1};

		\draw (-1,-8.35) -- (-1,-9.35);
		\draw (-1.5,-9.85) -- (-1,-9.35);
		\draw (-0.5,-9.85) -- (-1,-9.35);
		\draw (-0.5,-7.85) -- (-1,-8.35);
		\draw (-1.5,-7.85) -- (-1,-8.35);
		\filldraw (-1,-8.35) circle (1.7pt)   node[align=left, right] {$n\,$-1};
		\filldraw (-1,-9.35) circle (1.7pt)   node[align=left, right] {1};
		
		\draw [->, looseness=2, snake it ] (-2.6,-8.85) -- (-1.7,-8.85);
		\draw (-2.15,-9.85) node{(6)};

		\filldraw (-4,-7.85) circle (1pt);
		\filldraw (-3,-7.85) circle (1pt);
		\filldraw (-4.5,-9.85) circle (1pt);
		\filldraw (-3.5,-9.85) circle (1pt);

		\filldraw (-0.5,-7.85) circle (1pt);
		\filldraw (-1.5,-7.85) circle (1pt);
		\filldraw (-0.5,-9.85) circle (1pt);
		\filldraw (-1.5,-9.85) circle (1pt);


		\draw (3.5,-7.85) -- (3.5,-9.85);
		
		\filldraw (3.5,-8.35) circle (1.7pt)   node[align=left, right] {$n\,$-1};
		\filldraw (3.5,-9.35) circle (1.7pt)   node[align=left, right] {1};

		\draw (1,-9.35) -- (1,-9.85);
		\draw (1,-8.35) -- (1,-7.85);
		
		\draw (1,-8.35) to[out=200,in=160] (1,-9.35);
		\draw (1,-8.35) to[out=340,in=20] (1,-9.35);
		
		\filldraw (1,-8.35) circle (1.7pt)   node[align=left, right] {1};
		\filldraw (1,-9.35) circle (1.7pt)   node[align=left, right] {$n\,$-1};
		
		\draw [->, looseness=2, snake it ] (1.9,-8.85) -- (2.8,-8.85);
		\draw (2.35,-9.85) node{(7)};

		\filldraw (1,-7.85) circle (1pt);
		\filldraw (1,-9.85) circle (1pt);

		\filldraw (3.5,-7.85) circle (1pt);
		\filldraw (3.5,-9.85) circle (1pt);
		
	\end{tikzpicture}
	
	\caption{Combinatorial modifications of Reeb graphs related to change of the order of two consecutive critical points. On the right sides of vertices are their indices, where $1 \leq k \leq m \leq n-1$.
	}\label{figure:modifications}
\end{figure}

The above lemma provides a technique of \textbf{combinatorial modifications of Reeb graphs} by modifications of simple Morse functions. To be more precise, let $f\colon M \rightarrow \R$ be a simple Morse function on a manifold $M$, $v$ and $w$ be  adjacent vertices of $\reeb{f}$ and let $p$ and $p'$ be the critical points of $f$ corresponding to $v$ and~$w$, respectively, such that $\ind(p) \leq \ind(p')$ and $f(p) > f(p')$. Assume that~$W$, the connected component of $M^{\left[f(p')-\varepsilon, f(p)+\varepsilon\right]}$ containing $p$ and $p'$, contains no other critical points of $f$. Then $f$ can be modified on $W$ to a simple Morse function $g$ such that the Reeb graphs $\reeb{f|_W}$ and $\reeb{g|_W}$ are isomorphic to the graphs on the left and on the right side of a~suitable case in Figure \ref{figure:modifications}, respectively.

In fact, except the case (6), if vertices $v$ and $w$ are adjacent, we can always assume that $p$ and $p'$ are two consecutive critical points by rescaling $f$ on the triad corresponding to a small neighbourhood of the edge joining the vertices.

It is easily seen that for vertices $v$ and $w$ of degree $2$ or $3$ there are no cases other than those presented in Figure~\ref{figure:modifications}.

\begin{proof}
	By \cite[Theorems 4.1. and 4.4]{Milnor} there exists a simple Morse function $g\colon W\to \R$ on $(W,W_-,W_+)$ such that $g(p) < g(p')$ and with the same critical points and indices as $f$. We need only to show changes in the Reeb graph.	Let $q\colon W \to \reeb{g}$ be the quotient map, $k:=\ind(p) \leq \ind(p') =: m$ and let $v := q(p)$, $w:=q(p')$ be the vertices of $\reeb{g}$. The main properties, that we need, are connectedness of $\reeb{g}$, the number of connected components of $W_\pm$ and the correspondence from Proposition \ref{proposition:correspondece_degree_index}.
	
	Case $(1)$. Here $W_-$ and $W_+$ are connected. Since $W_{g(p)-\varepsilon}$ has the same number of connected components as $W_-$, $\degin(v)=1$. Similarly, $\degout(w)=1$. If $\degout(v)=2$, there would be $\degin(w)=2$, so $k=n-1$ and $m=1$, a~contradiction. Therefore $\degout(v)=1=\degin(w)$ and so $\reeb{g}$ is isomorphic to the graph on the right side of $(1)$.
	
	Case $(2)$. Assume that $m \neq 1$. For the same reason as above $\degin(v)=2$, and therefore $\degout(v)=1$. Thus $\degin(w)=\degout(w)=1$ and $\reeb{g}$ is as desired.
	In the same way we show the other cases when $k\neq m$. 
	
	An additional argument is needed when $k=m \in \{1,n-1\}$. Suppose that $k=m=1$ (for $n-1$ the proof is by duality). Consider a handle decomposition of $W$ corresponding to $f$. It consists of two $1$-handles attached by embeddings of $\es^0 \times \de^{n-1} = \de^{n-1} \sqcup \de^{n-1}$. By handle theory we can isotopically separate the images of these embeddings. Therefore the handles are attached to $W_- \times [0,1]$ along embeddings $\de^{n-1} \sqcup \de^{n-1} \to W_-\times \{1\}$. Since we can attach the handles in any order and because closed connected manifolds are homogeneous (i.e. for any two disjoint copies of $\de^{n-1}$ there is a self-diffeomorphism isotopic to the identity mapping one disc to the other), we can change the embeddings arbitrarily. Thus the handles can be attached to a required components of~$W_- \times \{1\}$. 
\end{proof}

\begin{lemma}\label{lemma:product_triad-cancellation_of_handles}
	If $(W,W_-,W_+)$ is a smooth triad with simple Morse function with exactly two critical points of index $0$ and $1$ ($n$ and $n-1$), then $(W,W_-,W_+)$ is a product triad, i.e. $W\cong W_- \times [0,1]$. Thus it admits a Morse function without critical points. 
\end{lemma}
\begin{proof}
	This is a special case of Cancellation of Handles \cite[Theorem 5.4]{Milnor}. 
\end{proof}

The above lemma gives us two additional modifications of Reeb graphs presented in Figure \ref{figure:modifications_product_triad}. In fact, they work in both ways. We can always assume that the critical points corresponding to vertices in (8) or (9) are consecutive by rescalling the function.

\begin{figure}[H]
	\centering
	
	\begin{tikzpicture}[scale=1.2]

		
		\draw (-3,2) -- (-3,1.5);
		\draw (-3.5,0.5) -- (-3,1.5);
		\draw (-2.5,0.5) -- (-3,1.5);
		\draw (-2.5,0.5) -- (-2.5,0);

		\draw (-0.5,2) -- (-0.5,0);
		
		\filldraw (-3.5,0.5) circle (1.7pt)   node[align=left, right] {0};
		\filldraw (-3,1.5) circle (1.7pt)   node[align=left, right] {1};
		
		\filldraw (-3,2) circle (1pt);
		\filldraw (-2.5,0) circle (1pt);
		
		\filldraw (-0.5,2) circle (1pt);
		\filldraw (-0.5,0) circle (1pt);

		\draw [->, looseness=2, snake it ] (-2,1.1) -- (-1,1.1);
		\draw [->, looseness=2, snake it ] (-1,0.9) -- (-2,0.9);
		\draw (-1.5,0) node{(8)};


		\draw (3,0) -- (3,0.5);
		\draw (3.5,1.5) -- (3,0.5);
		\draw (2.5,1.5) -- (3,0.5);
		\draw (3.5,1.5) -- (3.5,2);

		\draw (5.5,2) -- (5.5,0);
		
		\filldraw (2.5,1.5) circle (1.7pt)   node[align=left, right] {$n$};
		\filldraw (3,0.5) circle (1.7pt)   node[align=left, right] {$n-1$};
		
		\filldraw (3,0) circle (1pt);
		\filldraw (3.5,2) circle (1pt);
		
		\filldraw (5.5,2) circle (1pt);
		\filldraw (5.5,0) circle (1pt);

		\draw [->, looseness=2, snake it ] (4,1.1) -- (5,1.1);
		\draw [->, looseness=2, snake it ] (5,0.9) -- (4,0.9);
		\draw (4.5,0) node{(9)};

	\end{tikzpicture}
	\caption{Combinatorial modifications of Reeb graphs related to Cancellation of Handles.
	}\label{figure:modifications_product_triad}
\end{figure}
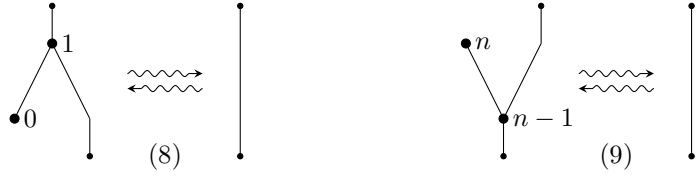

\begin{remark}
	Note that similar operations were introduced by Di Fabio and Landi~\cite{Fabio-Landi} for Reeb graphs of simple Morse functions on closed orientable surfaces. They called them elementary deformations. For manifolds of dimension at least three indices of critical points play important role in the modifications (1) -- (7) which can be used only if the index of the upper vertex is not greater than the index of the lower vertex. This causes that many of them do not work the other way around, contrary to the case of surfaces and elementary deformations, where the index of a critical point not being extremum is always equal to $1$. Another difference between combinatorial modifications and elementary deformations is in the occurrence of vertices of degree $2$ for the former operations. Reeb graphs of simple Morse functions on orientable surfaces have only vertices of degrees $1$ and $3$ (see \cite[Proposition 3.3.]{Michalak}). Their cycle rank is always equal to the genus of a surface, so the modification (7) also does not occur for them.

	Di Fabio and Landi used the elementary deformations to define an edit distance between Reeb graphs. For this purpose they added labels of vertices of Reeb graph with function values.
\end{remark}

Let $\Gamma$ be a~graph with good orientation. A vertex $v$ of degree 3 and of index 1 (respectively of index $n-1$) in $\Gamma$ is \textbf{branching}, if there exist two decreasing (respectively increasing) paths $\gamma, \delta\colon [0,1] \rightarrow \Gamma$  such that $\gamma(0) = \delta(0) = v$, their images are disjoint outside $v$ and both  $\gamma(1)$ and $\delta(1)$ are vertices of degree $1$ in~$\Gamma$.

\begin{remark}\label{remark:preserving_branching}
	Note that the modifcations we defined, except $(7)$, do not change the cycle rank of Reeb graph. One can see it directly or by Lemma \ref{lemma:number_of_cycles_Delta_3}. Moreover, except $(4)$, $(5)$, $(8)$ and $(9)$ they also do not change the property of being branching for a vertex of degree $3$. It is easy to check that the same is true for $(4)$ and $(5)$ if one of the vertices is branching and the other is not.
\end{remark}

For simplicity, a vertex of degree $1$ is called a \textbf{minimum} (\textbf{maximum}) if it has an outgoing edge (incoming edge).

\begin{lemma}\label{lemma:modification_to_one_minimum_maximum}
	Each simple Morse function $f\colon M \rightarrow \R$ on a closed $n$-manifold $M$, $n\geq 3$, can be modified using a finite number of combinatorial modifications to a~simple Morse function with exactly one minimum and maximum and with the same cycle rank of the Reeb graph.
\end{lemma}

\begin{proof}
	We take the lowest branching vertex $v$ of degree $3$ and index $1$ and we move it down using the modifications $(2)$, $(4)$ and $(6)$ so that it is adjacent to two minima. Since $v$ is branching, the modification $(6)$ can be used and we do not have to use $(7)$. Also $v$ will still be branching after $(4)$, since it is the lowest vertex with this property (see Remark \ref{remark:preserving_branching}). Then we use $(8)$ to remove $v$ and a one minimum and we repeat this procedure for each branching vertex of index 1.
	
	It is an easy exercise to show that the so-obtained Reeb graph has exactly one minimum. The proof for maxima is analogous. 
\end{proof}

By the above lemma $\reeb{M}$ can be attained by simple Morse functions with one minimum and maximum. By Lemma \ref{lemma:number_of_cycles_Delta_3} for such a function $f\colon M \to \R$ we have $\beta_1(\reeb{f}) = \frac{\Delta_3(f)}{2}$, where $\Delta_3(f)$ is the number of vertices of degree $3$ in $\reeb{f}$. Therefore
\begin{eqnarray*}
	\reeb{M} = \max \left\{\, \frac{\Delta_3(f)}{2}\ \ \begin{tabular}{|c}
		$f\colon M \to \R$ a simple Morse function \\
		with one minimum and maximum 
	\end{tabular} \right\}\!.
\end{eqnarray*}

\begin{definition}[{cf. \cite[Definition 1.3]{Fabio-Landi}, \cite{Kudryavtseva}}]
	The graph shown in Figure \ref{figure:canonical_form}~(a) is called \textbf{the canonical graph} (with a given cycle rank). A graph is in a \textbf{canonical form} if it is homeomorphic to the canonical graph and the homeomorphism adds vertices of degree $2$ only on non-cyclic edges.	
	\begin{figure}[H]
		\centering
		\begin{tikzpicture}[scale=0.9]
			
			
			\filldraw (0,0) circle (1.7pt);
			\filldraw (0,0.7) circle (1.7pt);
			\filldraw (0,1.7) circle (1.7pt);
			\filldraw (0,3) circle (1.7pt);
			\filldraw (0,4) circle (1.7pt);
			\filldraw (0,4.7) circle (1.7pt);
			\filldraw (0,5.7) circle (1.7pt);
			\filldraw (0,6.4) circle (1.7pt);
			
			\draw (0,0) -- (0,0.7);
			
			\draw (0,1.7) to[out=200,in=160] (0,0.7);
			\draw (0,1.7) to[out=340,in=20] (0,0.7);
			\draw (0,2) -- (0,1.7);
			\draw (0,2.7) -- (0,3);
			\draw (0,4) to[out=200,in=160] (0,3);
			\draw (0,4) to[out=340,in=20] (0,3);
			\draw (0,4.7) -- (0,4);
			\draw (0,5.7) to[out=200,in=160] (0,4.7);
			\draw (0,5.7) to[out=340,in=20] (0,4.7);
			\draw (0,6.4) -- (0,5.7);
			
			\draw (0,2.2) node { .};
			\draw (0,2.35) node { .};
			\draw (0,2.5) node { .};

			\draw (0,-0.5) node{(a)};
			

			\filldraw (3,0) circle (1.7pt);
			\filldraw (3,0.7) circle (1.7pt);
			\filldraw (3,1.7) circle (1.7pt);
			\filldraw (3,3) circle (1.7pt);
			\filldraw (3,4) circle (1.7pt);
			\filldraw (3,4.7) circle (1.7pt);
			\filldraw (3,5.7) circle (1.7pt);
			\filldraw (3,6.4) circle (1.7pt);
			
			\draw (3,0) -- (3,0.7);
			
			\draw (3,1.7) to[out=200,in=160] (3,0.7);
			\draw (3,1.7) to[out=340,in=20] (3,0.7);
			\draw (3,2) -- (3,1.7);
			\draw (3,2.7) -- (3,3);
			\draw (3,4) to[out=200,in=160] (3,3);
			\draw (3,4) to[out=340,in=20] (3,3);
			\draw (3,4.7) -- (3,4);
			\draw (3,5.7) to[out=200,in=160] (3,4.7);
			\draw (3,5.7) to[out=340,in=20] (3,4.7);
			\draw (3,6.4) -- (3,5.7);
			
			\draw (3,2.2) node { .};
			\draw (3,2.35) node { .};
			\draw (3,2.5) node { .};
			
			\filldraw (3,0.4) circle (1.7pt);
			\filldraw (3,1.9) circle (1.7pt);
			\filldraw (3,4.25) circle (1.7pt);
			\filldraw (3,4.5) circle (1.7pt);
			\filldraw (3,6.1) circle (1.7pt);

			\draw (3,-0.5) node{(b)};


			\filldraw (6,0) circle (1.7pt);
			\filldraw (6,0.7) circle (1.7pt);
			\filldraw (6,1.7) circle (1.7pt);
			\filldraw (6,3) circle (1.7pt);
			\filldraw (6,4) circle (1.7pt);
			\filldraw (6,4.7) circle (1.7pt);
			\filldraw (6,5.7) circle (1.7pt);
			\filldraw (6,6.4) circle (1.7pt);
			
			\draw (6,0) -- (6,0.7);
			
			\draw (6,1.7) to[out=200,in=160] (6,0.7);
			\draw (6,1.7) to[out=340,in=20] (6,0.7);
			\draw (6,2) -- (6,1.7);
			\draw (6,2.7) -- (6,3);
			\draw (6,4) to[out=200,in=160] (6,3);
			\draw (6,4) to[out=340,in=20] (6,3);
			\draw (6,4.7) -- (6,4);
			\draw (6,5.7) to[out=200,in=160] (6,4.7);
			\draw (6,5.7) to[out=340,in=20] (6,4.7);
			\draw (6,6.4) -- (6,5.7);
			
			\draw (6,2.2) node { .};
			\draw (6,2.35) node { .};
			\draw (6,2.5) node { .};
			
			\filldraw (6,0.4) circle (1.7pt);
			\filldraw (6,4.25) circle (1.7pt);
			\filldraw (6,6.1) circle (1.7pt);

			\filldraw (6.26,3.5) circle (1.7pt);
			\filldraw (5.74,1.2) circle (1.7pt);

			\draw (6,-0.5) node{(c)};

		\end{tikzpicture}
		
		\caption{(a) the canonical graph; (b)  graph in a canonical from; (c) graph not in a~canonical form.}\label{figure:canonical_form}
	\end{figure}
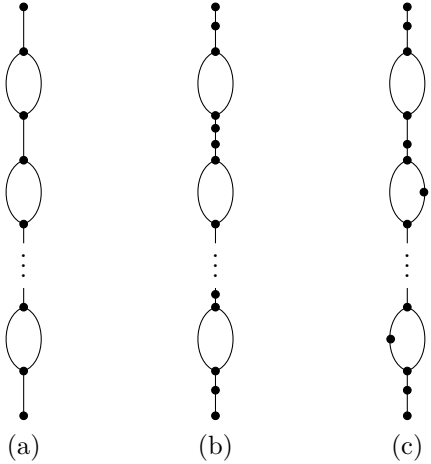
\end{definition}

The canonical graph with cycle rank equal to $g$ is the Reeb graph of a height function on a~closed orientable surface of genus $g$. A tree is in a canonical form if it is a~path.

The following proposition is similar to the ones for orientable surfaces which were shown by Kudryavtseva \cite[Theorem 1]{Kudryavtseva} and Fabio--Landi \cite[Lemma 2.6]{Fabio-Landi}.

\begin{proposition}\label{proposition:canonical_form}
	Let $f\colon M \rightarrow \R$ be a simple Morse function on a closed manifold $M$ of dimension $n\geq3$. Then $f$ can be modified using a finite number of combinatorial modifications to a simple Morse function whose Reeb graph has the same cycle rank and is in a~canonical form.
	
\end{proposition} 

\begin{proof}
	By Lemma \ref{lemma:modification_to_one_minimum_maximum} we may assume that $f$ has exactly one minimum and maximum. If $\reeb{f}$ is a tree, then it is in a canonical form, so assume that $\beta_1(\reeb{f})\geq 1$.
	
	First, we move down (move up) all vertices of degree $2$ and of index $1$ (of index $n-1$) in $\reeb{f}$ using the modifications $(1)$, $(2')$ and $(3)$ ($(1)$, $(2)$ and $(3')$) so that below the highest vertex of degree $2$ and of index $1$ (above the lowest vertex of degree $2$ and of index $n-1$) there will be only other such vertices and the minimum (the maximum).
	
	Let $v$ be the lowest vertex of degree $3$ and of index $1$ and let $w$ be the highest vertex of degree $3$ and of index $n-1$ which meets two different decreasing paths $\gamma$ and $\delta$ starting from $v$. On paths $\gamma$ and $\delta$ there are vertices of indices $2,\ldots,n-1$. We move all of them above $v$ using the modifications $(2)$ and $(6)$ (we do not use $(4)$ and $(7)$). We obtain a graph with a neighbourhood of $v$ and $w$ as on the left side of $(7)$.
	
	On the path from $w$ to the minimum there may be other vertices of degree $3$ and of index $n-1$. Let $u$ be the highest such a vertex. Using $(3)$ we move it up just below $w$. Now, the situation is as in Figure \ref{figure:omitting_cycle} $(i)$ and we perform the modifications $(5)$ and $(6)$ as in the figure. We repeat this procedure for all such vertices $u$. Then below $w$ there are only vertices of degree $2$ and the minimum.
	
	\begin{figure}[h]
		\centering
		
		\begin{tikzpicture}[scale=1.1]
			

			\draw (0,0.7) -- (0,0);
			\draw (0,0.7) -- (-0.5,1.4);
			\draw (-0.5,2.8) -- (-0.5,1.4);
			\draw (0,0.7) -- (0.5,1.4);
			\draw (0.5,2.1) to[out=200,in=160] (0.5,1.4);
			\draw (0.5,2.1) to[out=340,in=20] (0.5,1.4);
			\draw (0.5,2.1) -- (0.5,2.8);

			\filldraw (0,0.7) circle (1.7pt)   node[align=left, right] {$u$};
			\filldraw (0.5,1.4) circle (1.7pt)   node[align=left, right] {$w$};
			\filldraw (0.5,2.11) circle (1.7pt)   node[align=left, right] {$v$};

			\draw [->, looseness=2, snake it ] (1.4,1.5) -- (2.2,1.5);
			\draw (1.8,1.8) node{\scriptsize  (5)};
			\draw (0,-0.3) node{($i$)};
			
			
			\draw (3,0.7) -- (3,0);
			\draw (3,0.7) -- (3.5,1.4);
			\draw (3,2.1) -- (3.5,1.4);
			\draw (4,2.1) -- (3.5,1.4);
			\draw (4,2.1) -- (4,2.8);
			\draw (3,2.1) -- (3,2.8);
			\draw (3,2.1) to[out=200,in=160] (3,0.7);
			
			\filldraw (3.5,1.4) circle (1.7pt)   node[align=left, right] {$u$};
			\filldraw (3,0.7) circle (1.7pt)   node[align=left, right] {$w$};
			\filldraw (3,2.1) circle (1.7pt)   node[align=left, right] {$v$};
			
			\draw [->, looseness=2, snake it ] (4.4,1.5) -- (5.2,1.5);
			\draw (4.8,1.8) node{\scriptsize  (6)};
			\draw (3,-0.3) node{($ii$)};

			
			\draw (6,0.7) -- (6,0);
			\draw (6,1.4) -- (6,2.1);
			\draw (6.5,2.8) -- (6,2.1);
			\draw (5.5,2.8) -- (6,2.1);
			
			\draw (6,1.4) to[out=200,in=160] (6,0.7);
			\draw (6,1.4) to[out=340,in=20](6,0.7);

			\filldraw (6,2.1) circle (1.7pt)   node[align=left, right] {$u$};
			\filldraw (6,0.7) circle (1.7pt)   node[align=left, right] {$w$};
			\filldraw (6,1.4) circle (1.7pt)   node[align=left, right] {$v$};

			\draw (6,-0.3) node{($iii$)};
			
			
		\end{tikzpicture}
		\caption{Moving $u$ up above the cycle created by $v$ and $w$.}\label{figure:omitting_cycle}
	\end{figure}
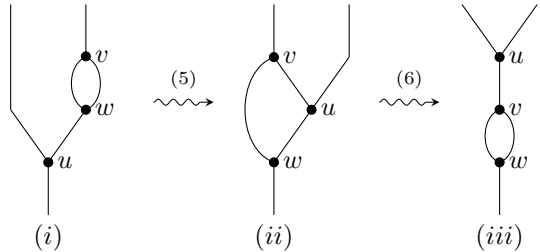
	
	Performing the entire above procedure for each vertex of degree $3$ and of index $1$ we obtain a graph in a canonical form. The cycle rank is unchanged. 
\end{proof}

\renewcommand*{\proofname}{Proof}

As a conclusion we get the main theorem of this section.

\begin{theorem}\label{theorem:each_number_of_cycles}
	Let $M$ be a closed manifold of dimension $n \geq 2$. For any number $0\leq k\leq \reeb{M}$ there exists a Morse function $g\colon M \rightarrow \R$ such that $\beta_1(\reeb{g}) = k$ and it can be simple if $M$ is not an orientable surface.
\end{theorem}

\begin{proof}
	For $n \geq 3$ let $f\colon M \rightarrow \R$ be a simple Morse function such that $\beta_1(\reeb{f})=\reeb{M}$ and $\reeb{f}$ is in a canonical form. By $(\reeb{M} -k)$-fold use of the modification $(7)$ we get a simple Morse function $g$ such that $\reeb{g}$ has cycle rank equal to $k$. 
	
	If $n=2$, then the statement follows by the direct constructions of functions (see \cite{Edelsbrunner} and \cite[Theorem 5.6]{Michalak}). The exception for orientable surfaces comes from the fact that the Reeb graph of a simple Morse function on a closed orientable surface of genus $g$ has always cycle rank equal to $g$ (see \cite{Edelsbrunner}).
	
\end{proof}



\begin{remark}
	In fact, the statement of the Theorem \ref{theorem:each_number_of_cycles} for orientable manifolds is a part of \cite[Theorem 13]{Gelbukh:Reeb_graph} which has been proven by Gelbukh.
\end{remark}


\section{Reeb number and corank of fundamental group}\label{section:reeb_number_and_corank}

\begin{definition}\label{definition:corank}
	Let $G$ be a finitely generated group. We say that $G$ has a \textbf{free quotient} of rank $r$, if there exists an epimorphism $G \rightarrow \F_r$ onto the free group on $r$ generators. The largest such a number $r$ we call the \textbf{corank} of $G$ and we denote it by $\corank(G)$.
\end{definition}

If $f\colon M \rightarrow \R$ is a smooth function with finitely many critical points then by \cite[Proposition 5.1]{KMS} the homomorphism $q_\# \colon \pi_1(M) \rightarrow \pi_1(\reeb{f}) \cong \F_{\beta_1(\reeb{f})}$ induced on fundamental groups by the quotient map is an epimorphism. Thus $\reeb{M} \leq \corank(\pi_1(M))$. In fact, the above inequality is equality, what follows from the next theorem.

\begin{theorem}\label{thm:equivalent_conditions}
	Let $M$ be a closed manifold of dimension $n\geq 2$ and let $r\geq 0$ be an integer. The following are equivalent:
	\begin{enumerate}[(a)]
		\item There exists a Morse function $g\colon M \rightarrow \R$ such that $\beta_1(\reeb{g}) = r$.
		\item The group $\pi_1(M)$ has a free quotient of rank $r$.
		\item There exist disjoint submanifolds $N_1,\ldots,N_r \subset M$ of codimension 1 with product neighbourhoods such that $M\setminus \bigcup_{i=1}^r N_i$ is connected.
	\end{enumerate}
	Moreover, if $M$ is not an orientable surface, then the function $g$ in (a) can be taken to be simple.
\end{theorem}

\begin{remark}
	The equivalence of conditions (b) and (c) has been described by Cornea \cite[Theorem 1]{Cornea} and for combinatorial manifolds by Jaco \cite[Theorem~2.1]{Jaco}. The main part of the below proof is showing that (c) implies (a).
\end{remark}

\begin{proof}
	The case $r=0$ is provided by Proposition \ref{proposition:ordered_has_tree}. For (a) implies (b), if $g \colon M \rightarrow \R$ is a Morse function and $\beta_1(\reeb{g}) = r$ then by \cite[Proposition 5.1]{KMS} the quotient map $q \colon M \rightarrow \reeb{g}$ induces the epimorphism $q_\# \colon \pi_1(M) \rightarrow \pi_1(\reeb{g}) \cong \F_r$.	
	The implication from (b) to (c) follows by \cite[Theorem 1]{Cornea}.

	It remains to prove that (c) implies (a). Let $P(N_i) \subset M$ be a closed product neighbourhood of  $N_i$ in $M$, i.e. $\ P(N_i)\cong N_i \times [-1,1]$, so small that $P(N_i) \cap P(N_j) = \emptyset$ for $i\neq j$. Denote by $N_i^\pm$ the submanifolds of $M$ correspoding to $N_i \times \{\pm1\}$. By the assumptions, the compact manifold $W := M \setminus \bigcup_{i=1}^{r} \Int P(N_i)$ is connected with boundary $\partial W = W_- \sqcup W_+$, where $W_\pm = \bigcup_{i=1}^r N_i^\pm$. Let $f\colon W \rightarrow [a,b]$ be an ordered Morse function of the triad $(W,W_-,W_+)$. By Lemma \ref{lemma:ordered_function_is_tree_with_connected_preimage} for $n\geq3$ there is a regular value $a < c < b$ such that $V := f^{-1}(c)$ is a connected submanifold of $W$ of codimension 1 (for $n=2$ we take arbitrary $c$ and then $V$ may not be connected). Let $$P(V) := f^{-1}([c-\varepsilon,c+\varepsilon]) \cong V \times [c-\varepsilon,c+\varepsilon]$$ be a~product neighbourhood of $V$ for small $\varepsilon >0$ and let $V_\pm := f^{-1}(c\pm\varepsilon)$.

	Manifolds $Q_- := f^{-1}([a,c-\varepsilon])$ and $Q_+ := f^{-1}([c+\varepsilon,b])$ have the boundary $\partial Q_\pm = V_\pm \sqcup W_\pm$. Take simple and ordered Morse functions $$g_- \colon Q_- \rightarrow [-2,-1] \text{ on the triad } (Q_-,\emptyset,\partial Q_-),$$
	$$g_+ \colon Q_+ \rightarrow [1,2] \text{ on the triad } (Q_+,\partial Q_+,\emptyset).$$

	Let us define functions
	$$h_i \colon P(N_i) \cong N_i \times [-1,1] \rightarrow [-1,1] \text{ by } h_i(x,t) =t,$$
	$$h_V \colon V \cong V \times [c-\varepsilon,c+\varepsilon] \rightarrow [-1,1] \text{ by } h_V(x,t) = \frac{t-(c-\varepsilon)}{\varepsilon} -1.$$
	
	Since $M = \bigcup_{i=1}^{r} P(N_i) \cup P(V) \cup Q_- \cup Q_+$, we define a  function $g\colon M \rightarrow [-2,2]$ which is the piecewise extension of the functions $g_\pm$, $h_i$ and $h_V$. It is well defined and it is a simple Morse function, because on the levels $[-2,-1]$ and $[1,2]$ the functions $g_\pm$ are simple and on the levels $[-1,1]$ there are no critical points.
	
	Let $q \colon M \rightarrow \reeb{g}$ be the quotient map. Since $g^{-1}(0) = V \sqcup \bigsqcup_{i=1}^{r} N_i$, all of points $t_i := q(N_i)$ lie on different edges in $\reeb{g}$. By the assumption $M \setminus \bigcup_{i=1}^r N_i$ is connected, so its image $\Gamma := q\left(M \setminus \bigcup_{i=1}^r N_i \right)$ is also connected. Since $\reeb{g} = \Gamma \cup t_1 \cup \ldots \cup t_r$, the space $\Gamma$ has the homotopy type of graph with $r$ edges less than $\reeb{f}$. Therefore $\beta_1(\reeb{g}) = \beta_1(\Gamma)+ r$, hence $\reeb{M} \geq \beta_1(\reeb{g}) \geq r$.
	
	Now, we may use Theorem \ref{theorem:each_number_of_cycles} to obtain cycle rank equal to $r$.
	
	If $n\geq 3$ we may also note that since $V$ is connected, $M\setminus \left(\bigcup_{i=1}^r N_i \cup V\right)$ is disconnected and because $\reeb{g_\pm}$ are trees by Proposition~\ref{proposition:ordered_has_tree}, the space $\Gamma$ is contractible and $\reeb{g}$ has indeed cycle rank equal to $r$. 
\end{proof}

A straightforward conclusion is the following equality which has been proven by Gelbukh \cite[Theorem 13]{Gelbukh:Reeb_graph} for orientable manifolds.

\begin{corollary}\label{corollary:corank=reeb_number}
	If $M$ is a closed manifold, then $\reeb{M} = \corank (\pi_1(M))$. 
\end{corollary}

\begin{corollary}
	Let $M$ and $N$ be closed manifolds of dimension $n\geq 2$. Then
	\begin{enumerate}[(a)]
		\item $\reeb{M\times N} = \max\{{\reeb{M},\reeb{N}}\}$,
		\item $\reeb{M\#N} = \reeb{M} + \reeb{N}$ if $n\geq 3$,
	\end{enumerate}
	where $\#$ denotes the connected sum operation.
\end{corollary}
\begin{proof}
	These statements follow from the analogous facts for the corank of fundamental group (\cite[Example 2 and 3]{Cornea} or \cite[Theorem 3.1]{Gelbukh:corank}, \cite[Theorem 3.2]{Jaco}).
\end{proof}

\begin{remark}
	The above equation for connected sum is also true if one of the surfaces is orientable, but it does not hold for non-orientable surfaces. Let $K = \R P^2 \# \R P^2$ be the Klein bootle. Then by \cite{Michalak} $\reeb{K} =1$, but $\reeb{\R P^2} =0$.
\end{remark}	

\begin{example}
	$\reeb{\#_{i=1}^g \es^1\times \es^{n-1}} = g$ and $\reeb{\#_{i=1}^g \es^1 \times \R P^{n-1}} = g$ for $n\geq 2$.
\end{example}


\section{Realization theorem}\label{section:realization}

\begin{definition}
	
	An oriented graph orientation-preserving homeomorphic to the graph in Figure \ref{figure:initial_graph} (with a given cycle rank) is in an \textbf{initial form} if the homeomorphism adds vertices of degree $2$ only on the two edges incident to vertices of degree $1$.
	
	\begin{figure}[h]
		\centering
		\begin{tikzpicture}[scale=0.85]

			\filldraw (0,0) circle (1.7pt);
			\filldraw (0,1) circle (1.7pt);
			\filldraw (0.5,1.5) circle (1.7pt);
			\filldraw (1,2) circle (1.7pt);
			\filldraw (1.5,2.5) circle (1.7pt);
			\filldraw (2,3) circle (1.7pt);
			\filldraw (2,4) circle (1.7pt);
			\filldraw (1.5,4.5) circle (1.7pt);
			\filldraw (1,5) circle (1.7pt);
			\filldraw (0.5,5.5) circle (1.7pt);
			\filldraw (0,6) circle (1.7pt);
			\filldraw (0,7) circle (1.7pt);
			
			\draw (0,0) -- (0,1);
			\draw (0.5,1.5) -- (0,1);
			\draw (0.5,1.5) -- (1,2);
			\draw (1.1,2.1) -- (1,2);
			\draw (1.4,2.4) -- (1.5,2.5);
			\draw (1.5,2.5) -- (2,3);

			\draw (0,7) -- (0,6);
			\draw (0.5,5.5) -- (0,6);
			\draw (0.5,5.5) -- (1,5);
			\draw (1.1,4.9) -- (1,5);
			\draw (1.4,4.6) -- (1.5,4.5);
			\draw (1.5,4.5) -- (2,4);

			\draw (2,4) to[out=200,in=160] (2,3);
			\draw (2,4) to[out=340,in=20] (2,3);

			\draw (1.5,4.5) to[out=225,in=135] (1.5,2.5);
			\draw (1,5) to[out=225,in=135] (1,2);
			\draw (0.5,5.5) to[out=225,in=135] (0.5,1.5);
			\draw (0,6) to[out=225,in=135] (0,1);

			\draw (1.34,2.34) node { .};
			\draw (1.25,2.25) node { .};
			\draw (1.16,2.16) node { .};

			\draw (1.34,4.66) node { .};
			\draw (1.25,4.75) node { .};
			\draw (1.16,4.84) node { .};
			
		\end{tikzpicture}
		
		\caption{The initial graph.}\label{figure:initial_graph}
	\end{figure}
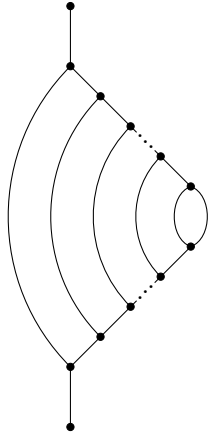
\end{definition}

The initial graph with cycle rank equal to $g$ occurs easily as the Reeb graph of a height function on an orientable surface of genus $g$. In fact, by \cite[Theorem 5.6]{Michalak} it~can be the Reeb graph of a Morse function on any closed surface with the Reeb number at least $g$.

\begin{proposition}\label{proposition:initial_graph}
	Let $M$ be a closed manifold of dimension $n\geq 3$. For any number $0 \leq k \leq \reeb{M}$ there exists a simple Morse function on $M$ whose Reeb graph is in an initial form and has cycle rank equal to $k$.
\end{proposition}

\begin{proof}
	In the proof of implication (c) $\implies$ (a) in Theorem \ref{thm:equivalent_conditions} for $r=k$ let $g_\pm$ be simple and ordered Morse functions on $Q_\pm$ with only one critical point being extremum. Then by Proposition \ref{proposition:ordered_has_tree} the Reeb graph $\reeb{g_-}$ (resp. $\reeb{g_+}$) is a~tree with one minimum (maximum) and all vertices of degree $3$ are of index $n-1$ (of index $1$). Thus $\reeb{g}$ (where $g\colon M \to \R$ is as in the proof of Theorem \ref{thm:equivalent_conditions}) has only two vertices of degree $1$ and cycle rank equal to $k$. We move up (move down) all vertices of degree $2$ in $\reeb{g_-}$ (in $\reeb{g_+}$ respectively). Therefore by using the modifications (4) and (5) on $g|_{Q_\pm} = g_\pm$ we can obtain a simple Morse function on $M$ whose Reeb graph is in an initial form. 
\end{proof}

We say that a function $f\colon M \to \R$ \textbf{realizes} a graph $\Gamma$ with good orientation if $\reeb{f}$ is orientation-preserving homeomorphic to $\Gamma$. If it is the case, $\Gamma$ is called \textbf{realizable} on $M$ by $f$.

\begin{remark}\label{remark:reverse_modifications_and_realizability}
	From now on, we will use combinatorial modifications for arbitrary graphs with good orientations, not only for Reeb graphs. Let us note that if a graph $\Gamma$ is realizable on $M$ by a simple Morse function, then any graph $\Gamma'$ obtained from $\Gamma$ by using combinatorial modifications is also realizable on $M$. Conversely, if we want to show a realization of $\Gamma'$ and we know that $\Gamma$ is obtained from $\Gamma'$ by using the reverse combinatorial modifications, it is sufficient to show a realizability of~$\Gamma$ by a simple Morse function. Recall that the modifications $(4)$, $(5)$, $(8)$ and $(9)$ are two-sided.
\end{remark}

\begin{theorem}\label{thm:realization}
	Let $M$ be a closed, connected $n$-dimensional manifold, $n\geq 2$, and $\Gamma$ be a finite oriented graph. There exists a Morse function $f\colon M \to \R$ such that $\reeb{f}$ is orientation-preserving homeomorphic to~$\Gamma$ if and only if $\Gamma$ has a good orientation and $\beta_1(\Gamma) \leq \reeb{M}$. Moreover, if $M$ is not an orientable surface and the maximum degree of a vertex in $\Gamma$ is not greater than $3$, then $f$ can be taken to be simple. 
\end{theorem}

\begin{proof}
	The case of surfaces is provided by \cite[Theorem 5.6]{Michalak}. 
	
	Let us assume that $n\geq 3$,
	Throughout the proof, we will define three additional combinatorial modifications of Reeb graphs ((10), (11) and (12)). The sketch of the proof is as follows. The first step is to reduce the considerations to graphs with vertices of degrees $1$, $2$ and $3$ (see Figure~\ref{figure:realization_from_high_degree_to_3_and_simple}). Then we will only work with simple Morse functions. Next we reduce the proof to the case of graphs whose all vertices of indegree $2$ are above vertices of outdegree $2$, as it is for the initial graph. In the last step we proceed by the induction on the number of vertices of degree~$1$. The crucial part of the proof is to show the induction step. In a graph $\Gamma$ we consider possible neighbourhoods of an edge $e$ incident to vertex $w$ of degree $3$ and to a vertex $v$ of degree $1$ (see Figure \ref{figure:cases_for_vertices_degree_1}). The problematic case is when $e$ is the only edge incoming to $w$ or outgoing from $w$. For this case if $w$ separates $\Gamma$ into three connected components, then we provide another modification of Reeb graphs which increases by $1$ the number of vertices of degree $1$ (see Figure~\ref{figure:proper_vertices}). There remains the case when $\Gamma \setminus \{w\}$ has two components for any vertices $v$ and $w$ as before. Assume that $v$ is a maximum vertex, $w$ is adjacent to $v$, $v'$ is a minimum joined with $v$ by a monotonic path and $w'$ is incident to $v'$. We show that it suffices to consider the case when all increasing paths from $v'$ to $v$ omit an edge $e$ incident to $w$ and an edge $e'$ incident to $w'$, where the edges $e$ and $e'$ are the same for all these paths. Next we provide a construction of specific function inducing the same good orientation as the original one for which the situation looks like in Figure~\ref{figure:removing_two_extrema_case_b} (a), i.e. the edge $e$ is so long that it ends below the chosen point $x$ on $e'$. Then we introduce the last modification of Reeb graphs which increases by $2$ the number of vertices of degree $1$ (see Figure~\ref{figure:removing_two_extrema_case_b}), and it completes the proof.
	
	\textbf{Step 1}. We first reduce the problem to graphs whose maximum degree is not greater than $3$ by introducing combinatorial modification number~(10). Let $\Gamma'$ be a graph which is obtained from $\Gamma$ by substituting a~small neighbourhood of each vertex $v$ in $\Gamma$ such that $\deg(v) \geq 4$ into a suitable one denoted by $S(v)$ as in Figure \ref{figure:realization_from_high_degree_to_3_and_simple}. Then $\Gamma'$ is a graph with the maximum degree not greater than $3$. If there exists a simple Morse function $f$ on $M$ which realizes $\Gamma'$, then identifying $S(v)$ with a subset of $\reeb{f}$  there are $\degout(v)-1$ vertices of index $n-1$ below $\degin(v)-1$ vertices of index $1$ in $S(v)$, all of degree $3$. Any vertex of degree $2$ in $S(v)$ can be moved outside $S(v)$ by using modifications of Reeb graps. By \cite[Theorem 4.1, 4.2 Extension and 4.4]{Milnor} we can rearrange the corresponding critical points to a single critical level of a new Morse function. Then the vertices in $S(v)$ collapse to a single vertex with neighbourhood homeomorphic to the neighbourhood of $v$ in $\Gamma$. If we perform this for all $S(v)$, then the obtained Morse function will realize $\Gamma$.
	
	\begin{figure}[h]
		\centering
		\begin{tikzpicture}[scale=1.2]

			\filldraw (0,0) circle (1.7pt)  node[align=left, right] {$v$};

			\draw (0,0) -- (-0.2,1.5);
			\draw (0,0) -- (0.6,1.5);
			\draw (0,0) -- (-0.6,1.5);
			\draw (0,0) -- (0.6,-1.5);
			\draw (0,0) -- (0,-1.5);
			\draw (0,0) -- (-0.3,-1.5);
			\draw (0,0) -- (-0.6,-1.5);
			
			\draw (0.1,1.2) node { .};
			\draw (0.2,1.2) node { .};
			\draw (0.3,1.2) node { .};
			
			\draw (0.15,-1.2) node { .};
			\draw (0.25,-1.2) node { .};
			\draw (0.35,-1.2) node { .};
			
			\draw [->, looseness=2, snake it ] (2,0) -- (1,0);
			
			\filldraw (3,-1) circle (1.7pt);
			\filldraw (3.5,-0.5) circle (1.7pt);
			\filldraw (4,0) circle (1.7pt);
			\filldraw (4.5,0.5) circle (1.7pt);
			\filldraw (5,1) circle (1.7pt);
			
			\draw (3,-1) -- (3,-1.5);
			\draw (5,1) -- (5,1.5);
			
			\draw (3,-1) -- (3.5,-0.5);
			\draw (4,0) -- (3.5,-0.5);
			\draw (4,0) -- (4.5,0.5);
			\draw (5,1) -- (4.5,0.5);
			
			\draw (3,-1) -- (2.5,-0.5);
			\draw (3,0) -- (3.5,-0.5);
			\draw (4,0) -- (4.5,-0.5);
			\draw (4.5,0.5) -- (5,0);
			\draw (5.5,0.5) -- (5,1);
			
			\draw (2.5,1.5) -- (2.5,-0.5);
			\draw (3,1.5) -- (3,0);
			\draw (4.5,-1.5) -- (4.5,-0.5);
			\draw (5,-1.5) -- (5,0);
			\draw (5.5,-1.5) -- (5.5,0.5);

			\draw (4,1) node {$S(v)$}; 
			
			
			\draw (3,-0.7) node { .};
			\draw (3.1,-0.6) node { .};
			\draw (3.2,-0.5) node { .};
			
			
			\draw (4.8,0.5) node { .};
			\draw (4.9,0.6) node { .};
			\draw (5,0.7) node { .};

			\draw (2.75,1.85) node {\footnotesize $\degout(v)-1$}; 
			\draw (5,-1.85) node {\footnotesize{$\degin(v)-1$}}; 
			
			\draw [decorate,decoration=brace]
			(2.4,1.57) -- (3.1,1.57);
			\draw [decorate,decoration=brace]
			(5.6,-1.57) -- (4.4,-1.57);

			\draw (1.5,-1.8) node{(10)};

		\end{tikzpicture}
		
		\caption{The combinatorial modification (10). If the Reeb graph of a simple Morse function has $S(v)$ as a subspace, then it can be modified to the Reeb graph of a Morse function in which $S(v)$ corresponds to a~neighbourhood of $v$ in $\Gamma$. }\label{figure:realization_from_high_degree_to_3_and_simple}
	\end{figure}
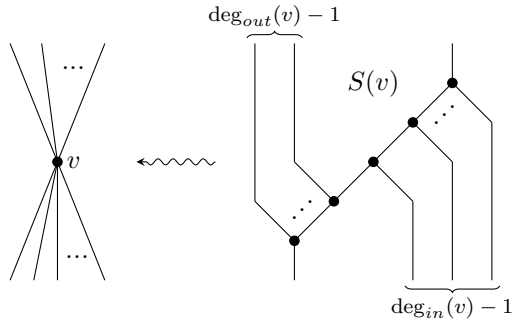
	
	Therefore we may assume that $\Gamma$ has no vertices of degrees other than $1$, $2$ and $3$. We will show that $\Gamma$ can be realized on $M$ by a simple Morse function. We may ignore vertices of degree $2$ since we are interested in a~homeomorphism type of graphs and we can always move them in a suitable way using combinatorial modifications. Thus assume that $\Gamma$ has only vertices of degrees $1$ or $3$.
	
	If there is no a vertex with indegree $2$ in $\Gamma$ below a vertex with outdegree $2$, then $\Gamma$ is called \textbf{primitive}. For example, a graph in an initial form is primitive.
	
	\textbf{Step 2}. We use the reverse modification (6) and modifications (4) and (5) (which are two-sided) on $\Gamma$ to make it primitive. Thus we reduced the problem to primitive graphs (see Remark \ref{remark:reverse_modifications_and_realizability}).

	\textbf{Step 3}. Assume $\Gamma$ to be primitive. We proceed by induction on the number of vertices of degree $1$. For the base case, suppose that $\Gamma$ has only one minimum and maximum. By Proposition \ref{proposition:initial_graph} there exists a simple Morse function $g$ on $M$ whose Reeb graph is in an initial form and has cycle rank equal to $\beta_1(\Gamma)$. By Proposition \ref{proposition:correspondece_degree_index} and Lemma \ref{lemma:number_of_cycles_Delta_3} both $\Gamma$ and $\reeb{g}$ has $\beta_1(\Gamma)$ vertices of outdegree $2$ and $\beta_1(\Gamma)$ vertices of indegree $2$. It is easily seen that by using the modifications (4) and (5) on $\reeb{g}$  we can reorder them to produce a simple Morse function which realizes $\Gamma$ as the Reeb graph.
	
	Now, let $v$ be a vertex of degree $1$ in $\Gamma$, $e$ be the edge incident to $v$ and let $w$ be the second vertex incident to $e$. If $w$ has degree $1$, then $\Gamma$ is the tree on two vertices and this case is provided by the base case. Hence we may assume that $\deg(w)=3$. We distinguish the following cases for vertices of degree $1$ in~$\Gamma$:
	\begin{enumerate}[(a)]
		\item $e$ is not the only edge which incomes to (or outgoes from) $w$,
		\item $e$ is the only edge incoming to (or outgoing from) $w$ and:
		\begin{enumerate}[(b1)]
			\item $\Gamma \setminus \{w\}$ has three connected components,
			\item $\Gamma \setminus \{w\}$ has two components.
		\end{enumerate}
	\end{enumerate}

	In the case (a) by Lemma \ref{lemma:product_triad-cancellation_of_handles} and by modifications (8) and (9) we can reduce $\Gamma$ to a graph without $v$.
	
	\begin{figure}[h]
		\centering
		\begin{tikzpicture}[scale=1]
			
			
			\filldraw (0,0) circle (1.7pt) node[align=left, left] {$w$};
			\filldraw (-0.5,-1) circle (1.7pt) node[align=left, left] {$v$};

			\draw (0,0) -- (0,1);
			\draw (0,0) -- (-0.5,-1);
			\draw (0,0) -- (0.5,-1);
			
			\draw (-0.15,-0.6) node{$e$};

			\filldraw (2,0) circle (1.7pt) node[align=left, left] {$w$};
			\filldraw (1.5,1) circle (1.7pt) node[align=left, left] {$v$};

			\draw (2,0) -- (2,-1);
			\draw (2,0) -- (1.5,1);
			\draw (2,0) -- (2.5,1);
			
			\draw (1.85,0.6) node{$e$};

			\draw (1.2,-1.2) node{(a)};


			\filldraw (5,0) circle (1.7pt) node[align=left, left] {$w$};
			\filldraw (5,1) circle (1.7pt) node[align=left, left] {$v$};

			\draw (5,0) -- (5,1);
			\draw (5,0) -- (4.5,-1);
			\draw (5,0) -- (5.5,-1);
			
			\draw (5.15,0.5) node{$e$};

			\filldraw (7,0) circle (1.7pt) node[align=left, left] {$w$};
			\filldraw (7,-1) circle (1.7pt) node[align=left, left] {$v$};

			\draw (7,0) -- (7,-1);
			\draw (7,0) -- (7.5,1);
			\draw (7,0) -- (6.5,1);
			
			\draw (7.15,-0.5) node{$e$};

			\draw (6.2,-1.2) node{(b)};
			
		\end{tikzpicture}
		
		\caption{Possible configurations of the vertices $v$ and $w$ in $\Gamma$.}\label{figure:cases_for_vertices_degree_1}
	\end{figure}
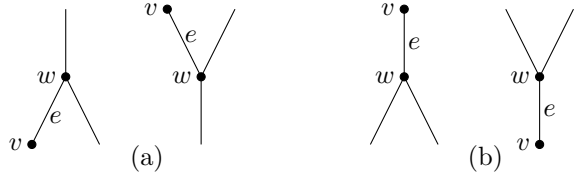

	For the case (b1) suppose that $e$ is an edge incoming to $w$ (the second case when $e$ outgoes from $w$ is analogous) Let $u_1$ and $u_2$ be other vertices adjacent to $w$ and let $\Gamma_1$, $\Gamma_2$ be connected components (except $\{v\}$) of a graph obtained from $\Gamma$ by removing $w$ and incident edges. Define $\Gamma'$ as an oriented graph obtained from $-\Gamma_1$ (i.e. $\Gamma_1$ with reverse orientation) and from $\Gamma_2$ by joining $u_1$ and $u_2$ by an edge $e'$. Figure \ref{figure:proper_vertices} shows the situation schematically. It has one vertex of degree $1$ less than $\Gamma$, so by the induction hypothesis and Step 2. (since $\Gamma'$ may not be primitive) there exists a simple Morse function $f'$ on $M$ that realizes~$\Gamma'$.
	
	\begin{figure}[H]
		\centering
		\begin{tikzpicture}[scale=1.1]

			\filldraw (0,0) circle (1.7pt) node[align=left, left] {$w$};
			\filldraw (0,-1) circle (1.7pt) node[align=left, left] {$v$};
			\filldraw (-0.5,1) circle (1.7pt) node[align=left, left, yshift=-0.4em] {$u_1$};
			\filldraw (0.5,1) circle (1.7pt) node[align=left, right, yshift=-0.4em] {$u_2$};

			\draw (0,0) -- (0,-1);
			\draw (0,0) -- (0.5,1);
			\draw (0,0) -- (-0.5,1);
			
			\draw (0.15,-0.5) node{$e$};
			
			\draw (-0.5,1.4) circle [radius=0.4] node[align=left] {$\Gamma_1$};
			\draw (0.5,1.4) circle [radius=0.4] node[align=left] {$\Gamma_2$};

			\draw (0,-1.5) node{$\Gamma$};

			\draw [->, looseness=2, snake it ] (2,0) -- (1,0);
			
			\draw (3,-0.4) -- (3,1);
			\draw (3.2,0.3) node{$e'$};
			
			\filldraw (3,-0.4) circle (1.7pt) node[align=left, left, yshift=0.4em] {$u_1$};
			\filldraw (3,1) circle (1.7pt) node[align=left, left, yshift=-0.4em] {$u_2$};

			\draw (3,-0.8) circle [radius=0.4] node[align=left] {$-\Gamma_1$};
			\draw (3,1.4) circle [radius=0.4] node[align=left] {$\Gamma_2$};

			\draw (3,-1.5) node{$\Gamma'$};

			\draw (1.5,-1.8) node{(11)};

		\end{tikzpicture}
		
		\caption{Construction of $\Gamma'$ from $\Gamma$. Realization of $\Gamma'$ as the Reeb graph of a simple Morse function implies realization of $\Gamma$, which leads to combinatorial modification number (11).}\label{figure:proper_vertices}
	\end{figure}
	
	Let $[a,b]$ be a small interval contained in $e'$ in $\reeb{f'}$ and let $G_1$, $G_2$ be the two components of $\reeb{f'} \setminus (a,b)$ such that $G_i$ corresponds to $\Gamma_i$. Divide $M$ into three submanifolds $Q_i = q^{-1}(G_i)$, $i=1,2$, and $W=q^{-1}([a,b])$ (where $q\colon M \to \reeb{f'}$ is the quotient map). The functions $f'|_{Q_1}$, $f'|_{Q_2}$ and $f'|_W$ are funtions on the smooth triads $(Q_1,\varnothing,W_-)$, $(Q_2,W_+,\varnothing)$ and $(W,W_-,W_+)$ respectively, where $W_- = q^{-1}(a)$ and $W_+ = q^{-1}(b)$ are connected submanifolds. Denote by $c$ and $d$ the levels of $f'$ for $W_-$ and $W_+$, respectively. Let $g\colon Q_1 \to \R$ be defined by $g(x)=-f'|_{Q_1}(x) + c+d$, which is a function on the triad $(Q_1,W_-,\varnothing)$ such that $g(W_-) = \{d\}$. We also need an ordered simple Morse function $h\colon W \to [d-\varepsilon,d]$ on the triad $(W,\varnothing,W_-\sqcup W_+)$ with only one critical point being extremum (here it is a minimum). By Propositions \ref{proposition:correspondece_degree_index} and \ref{proposition:ordered_has_tree} the Reeb graph $\reeb{h}$ is homeomorphic to a~neighbourhood of $w$ in $\Gamma$. Now, define a Morse function $f$ on $M$ which is the piecewise extension of $g$, $f'|_{Q_2}$ and $h$. It follows from the construction that $f$ realizes $\Gamma$. Since each component of a~level set of $f$ contains at most one critical point, $f$ can be taken to be simple. It is the combinatorial modification number (11).
	
	Now, suppose that the only vertices of degree $1$ in $\Gamma$ are from the case (b2). Suppose that $\Gamma$ has at least two maxima (a proof for minima is analogous).
	
	Let $v$ be a maximum vertex, $w$ vertex with indegree $2$ which is adjacent to $v$ and let $v'$ be a minimum joined with $v$ by a monotonic path $\tau$. Since $\Gamma$ has no vertices from the case (a), $v'$ is adjacent to a vertex $w'$ with outdegree $2$. Using the modifications (4) and (5) one can move out all vertices on $\tau$ between $w$ and $w'$. Let $x$ and $y$ be points on the edges incident to $w'$ and $w$, respectively, as in Figure~\ref{figure:removing_increasing_paths}~(a).
	
	Suppose that there exists an increasing path $\gamma$ from $x$ to $y$. Since $\Gamma$ has more than two vertices of degree $1$, there exists a vertex of degree $3$ on $\gamma$. We have the following two cases:
	\begin{enumerate}[(b2-I)]
		\item there are both the types of vertices of degree $3$ on $\gamma$,
		\item there is no vertex with outdegree $2$ or no vertex with indegree $2$ on $\gamma$.
	\end{enumerate}
	
	For the case (b2-I) let $z$ and $z'$ be vertices on $\gamma$ adjacent to $w$ and $w'$, respectively. Since $\Gamma$ is primitive, $\degin(z)=2=\degout(z')$. Use the modifications (4) and (5) to move out all vertices on $\gamma$ leaving only $z$ and $z'$, as in Figure \ref{figure:removing_increasing_paths} (b). Now, let us again use (4) to move $z$ on the second edge incident to $w$ and (5) to move $z'$ on the second edge incident to $w'$, as in Figure \ref{figure:removing_increasing_paths} (c). Thus we reduced the number of increasing paths from $x$ to $y$.
	
	For the case (b2-II) assume that there is no vertex with indegree $2$ on $\gamma$ (the second case is analogous). Let $z$ be the vertex adjacent to $w'$ with outdegree $2$. As in the previous case, we move out all vertices on $\gamma$ other than $z$ (all of them have outdegree $2$) and now $z$ is adjacent to $w$ and $w'$. Use (5) to move $z$ on the second edge incident to $w'$. Figure \ref{figure:removing_increasing_paths} (d) shows the situation.

	\begin{figure}[H]
		\centering
		\begin{tikzpicture}[scale=1.2]

			
			\filldraw (0,1) circle (1.7pt) node[align=left, left] {$w$};
			\filldraw (0,1.5) circle (1.7pt) node[align=left, left] {$v$};
			\filldraw (0,-1) circle (1.7pt) node[align=left, left, yshift=-0.2em] {$w'$};
			\filldraw (0,-1.5) circle (1.7pt) node[align=left, left] {$v'$};

			\draw (0,1) -- (0,1.5);
			\draw (0,-1) -- (0,-1.5);
			\draw (0,-1) -- (0.5,-0.75);
			\draw (0,1) -- (0.5,0.75);
			
			\draw (0,-1) to[out=150,in=210] (0,1);

			\filldraw (0.25,-0.875) circle (1pt) node[align=left, right,yshift=-0.2em] {$x$};
			\filldraw (0.25,0.875) circle (1pt) node[align=left, right,yshift=0.2em] {$y$};

			\draw (0.7,-0.65) node { .};
			\draw (0.8,-0.6) node { .};
			\draw (0.9,-0.55) node { .};
			
			\draw (0.7,0.65) node { .};
			\draw (0.8,0.6) node { .};
			\draw (0.9,0.55) node { .};

			\draw (0,-2.2) node{(a)};

			
			\filldraw (2.5,1) circle (1.7pt) node[align=left, left] {$w$};
			\filldraw (2.5,1.5) circle (1.7pt) node[align=left, left] {$v$};
			\filldraw (2.5,-1) circle (1.7pt) node[align=left, left, yshift=-0.2em] {$w'$};
			\filldraw (2.5,-1.5) circle (1.7pt) node[align=left, left] {$v'$};
			
			\draw (2.5,1) -- (2.5,1.5);
			\draw (2.5,-1) -- (2.5,-1.5);
			\draw (2.5,-1) -- (3.5,-0.5);
			\draw (2.5,1) -- (3.5,0.5);
			
			\draw (2.5,-1) to[out=150,in=210] (2.5,1);

			\filldraw (2.75,-0.875) circle (1pt) node[align=left, right,yshift=-0.4em, xshift = -0.2em] {$x$};
			\filldraw (2.75,0.875) circle (1pt) node[align=left, right,yshift=0.5em, xshift=-0.4em] {$y$};

			\filldraw (3,-0.75) circle (1.7pt) node[align=left, left, yshift=0.5em, xshift=0.3em] {$z'$};
			\filldraw (3,0.75) circle (1.7pt) node[align=left, left, yshift=-0.5em, xshift=0.2em] {$z$};
			
			\draw (3,-0.75) to[out=60,in=300] (3,0.75);

			\draw (3.6,-0.45) node { .};
			\draw (3.7,-0.4) node { .};
			\draw (3.8,-0.35) node { .};
			
			\draw (3.6,0.45) node { .};
			\draw (3.7,0.4) node { .};
			\draw (3.8,0.35) node { .};

			\draw (2.5,-2.2) node{(b)};
			
			
			\filldraw (5.5,1) circle (1.7pt) node[align=left, left] {$w$};
			\filldraw (5.5,1.5) circle (1.7pt) node[align=left, left] {$v$};
			\filldraw (5.5,-1) circle (1.7pt) node[align=left, left, yshift=-0.2em] {$w'$};
			\filldraw (5.5,-1.5) circle (1.7pt) node[align=left, left] {$v'$};
			\filldraw (5,-0.5) circle (1.7pt) node[align=left, left, yshift=-0.1em] {$z'$};		
			\filldraw (5,0.5) circle (1.7pt) node[align=left, left] {$z$};

			\draw (5.5,1) -- (5.5,1.5);
			\draw (5.5,-1) -- (5.5,-1.5);
			\draw (5.5,-1) -- (6,-0.75);
			\draw (5.5,1) -- (6,0.75);
			\draw (5.5,1) -- (5.,0.5);
			\draw (5.5,-1) -- (5,-0.5);
			
			\draw (5,-0.5) to[out=150,in=210] (5,0.5);
			\draw (5,-0.5) to[out=30,in=330] (5,0.5);

			\filldraw (5.75,-0.875) circle (1pt) node[align=left, right,yshift=-0.2em] {$x$};
			\filldraw (5.75,0.875) circle (1pt) node[align=left, right,yshift=0.2em] {$y$};
			
			\draw (6.2,-0.65) node { .};
			\draw (6.3,-0.6) node { .};
			\draw (6.4,-0.55) node { .};
			
			\draw (6.2,0.65) node { .};
			\draw (6.3,0.6) node { .};
			\draw (6.4,0.55) node { .};

			\draw (5.5,-2.2) node{(c)};

			
			\filldraw (8,1) circle (1.7pt) node[align=left, left] {$w$};
			\filldraw (8,1.5) circle (1.7pt) node[align=left, left] {$v$};
			\filldraw (8,-1) circle (1.7pt) node[align=left, left, yshift=-0.2em] {$w'$};
			\filldraw (8,-1.5) circle (1.7pt) node[align=left, left] {$v'$};
			\filldraw (7.5,-0.5) circle (1.7pt) node[align=left, left] {$z$};

			\draw (8,1) -- (8,1.5);
			\draw (8,-1) -- (8,-1.5);
			\draw (8,-1) -- (7.5,-0.5);
			\draw (8,-1) -- (8.5,-0.75);
			
			\draw (7.5,-0.5) to[out=105,in=210] (8,1);
			\draw (7.5,-0.5) to[out=345,in=330] (8,1);

			\filldraw (8.25,-0.875) circle (1pt) node[align=left, right,yshift=-0.2em] {$x$};
			\filldraw (8.22,0.7) circle (1pt) node[align=left, right,yshift=0.2em] {$y$};

			\draw (8.7,-0.65) node { .};
			\draw (8.8,-0.6) node { .};
			\draw (8.9,-0.55) node { .};

			\draw (8,-2.2) node{(d)};

		\end{tikzpicture}
		
		\caption{Reducing the number of increasing paths from $x$ to $y$.}
		\label{figure:removing_increasing_paths}
	\end{figure}
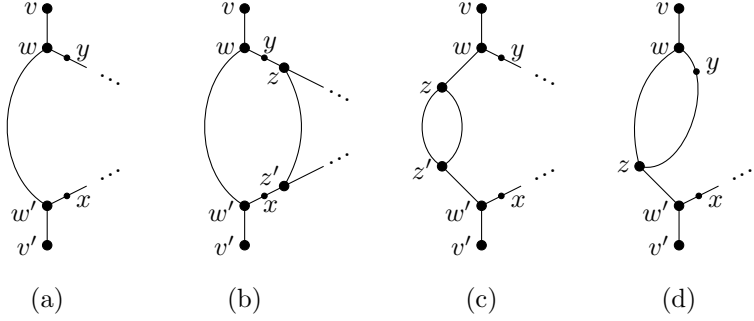
	
	For two points $p$ and $p'$ in $\Gamma$ denote by $\IP(p,p')$ ($\DP(p,p')$ respectively) the subset of $\Gamma$ consisting of images of all increasing (decreasing) paths from $p$ to $p'$. Similarly, denote by $\IP(p)$ ($\DP(p)$) the subset consisting of images of all increasing (decreasing) paths starting at $p$.
	
	Performing the above procedures for each increasing path from $x$ to $y$ we obtain a situation such that $\IP(x,y)=\varnothing$ (equivalently, $x\not\in \DP(y)$) and that $\IP(w',w) \setminus \{w',w\}$ is a connected component of $\Gamma \setminus \{w',w\}$. It is clear that if the case (b2-II) was occured, then $v'$ and $w'$ (or $v$ and $w$) are from the case (b1), so it gives a realization of $\Gamma$. Therefore we may assume that we only used a procedure from (b2-I).

	Let $g$ be a continuous function on $\Gamma$ inducing its good orientation and let $b = g(x)$ and $g(\DP(y)) = [c,d]$, where $g(y) =d$. Let $y'$ be a point near $y$ on the same edge such that $g(y') = d' < d$. We want to construct a new function $g'$ on $\Gamma$ such that $g'(y') < g'(x)$ and which induces the same orientation. Let $\varepsilon >0$, $a < b$ and $h\colon [c,d] \to [a-\varepsilon,d]$ be an orientation-preserving homeomorphism of intervals such that $[d',d]$ is mapped to $[a,d]$. Denote by $E_y$ the set of edges $e$ in $\Gamma$ whose closure $\cl(e)$ intersects $\DP(y)$ in an only one end. If $e\in E_y$, then $e$ is incident to a unique vertex $z \in \DP(y)$. Since $e \not\subset \DP(y)$, $e$ outgoes from $z$. Let $h_e$ be an orientation-preserving homeomorphism of $[g(z),d_e]$ onto $[h(g(z)),d_e]$, where $g(\cl(e)) = [g(z), d_e]$. We define a continuous function $g'\colon \Gamma \to \R$ by
	
	\begin{eqnarray*}
		g' (s) =  \begin{cases}
			h(g(s))	        &\text{\ \ \ \ if $s \in \DP(y)$,}\\
			h_e(g(s))	        &\text{\ \ \ \ if $s \in e \in E_y $,}\\
			g(s)	        &\text{\ \ \ \ in other cases.}
		\end{cases}
	\end{eqnarray*}
	
	It is clear that the orientation induced by $g'$ is the same as $g$ and that $g'(y') = a < b = g'(x)$.
	
	Now, let $z'$ and $z$ be points on the edges contained in $\IP(w',w)$ and incident to $w'$ and $w$, respectively (see Figure \ref{figure:removing_two_extrema_case_b} (a)). Define an oriented graph $\Gamma'$ obtained from $\Gamma$ by:
	\begin{enumerate}	
		\item  removing an open neighbourhood of an edge incident to $w$ and $v$ (to $w'$ and $v'$) with $z$ and $y'$ ($z'$ and $x$) as boundary points,
		\item taking $A = \IP(z',z)$ with the reverse orientation,
		\item joining $y'$ with $z$ and $z'$ with $x$ by a segment.
	\end{enumerate}
	
	Figure \ref{figure:removing_two_extrema_case_b} (b) shows this construction schematically. It is evident that $\Gamma'$ has a~good orientation and has two vertices of degree $1$ less than $\Gamma$. Thus Step~2. and induction hypothesis give us a realization of $\Gamma'$ by a simple Morse function~$f'$.
	\begin{figure}[h]
		\centering
		\begin{tikzpicture}[scale=1.3]

			
			\filldraw (0,1) circle (1.7pt) node[align=left, left] {$w$};
			\filldraw (0,1.5) circle (1.7pt) node[align=left, left] {$v$};
			\filldraw (0,-1) circle (1.7pt) node[align=left, left, yshift=-0.2em] {$w'$};
			\filldraw (0,-1.5) circle (1.7pt) node[align=left, left] {$v'$};

			\draw (0,1) -- (0,1.5);
			\draw (0,-1) -- (0,-1.5);
			\draw (0,-1) -- (0.5,-0.5);
			\draw (0,1) -- (3,-2);
			\draw (0,-1) -- (-0.35,-0.65);
			\draw (0,1) -- (-0.35,0.65);
			ä
			
			\filldraw (0.25,-0.75) circle (1pt) node[align=left, right,yshift=-0.2em] {$x$};
			\filldraw (0.25,0.75) circle (1pt) node[align=left, right,yshift=0.2em] {$y$};
			\filldraw (2.7,-1.7) circle (1pt) node[align=left, right,yshift=0.2em] {$y'$};
			
			\filldraw (-0.35,-0.65) circle (1pt) node[align=left, left,yshift=-0.3em, xshift=0.1em] {$z'$};
			\filldraw (-0.35,0.65) circle (1pt) node[align=left, left,yshift=0.4em, xshift=0.2em] {$z$};
			
			\draw (-0.35,0) node{$A$};
			\draw (-0.35,0) circle [x radius=0.5cm, y radius=0.65cm];
			
			\draw (3.1,-2.1) node { .};
			\draw (3.2,-2.2) node { .};
			\draw (3.3,-2.3) node { .};
			
			\draw (0.6,-0.4) node { .};
			\draw (0.7,-0.3) node { .};
			\draw (0.8,-0.2) node { .};

			\draw (0,-2) node{$\Gamma$};
			\draw (0,-2.7) node{(a)};

			\draw [->, looseness=2, snake it ] (4,0) -- (3,0);


			\draw (6.4,-1.4) -- (7,-2);
			
			\draw (5.5,-0.5) .. controls (5,0) .. (5.3,0.3) -- (5.5,0.5);
			
			\filldraw (5.3,0.3) circle (1pt) node[align=left, left,yshift=0.2em] {$x$};
			\filldraw (6.7,-1.7) circle (1pt) node[align=left, right,yshift=0.2em] {$y'$};
			
			\filldraw (5.5,-0.5) circle (1pt) node[align=left, left,yshift=-0.3em, xshift=0.1em] {$z'$};
			\filldraw (6.4,-1.4) circle (1pt) node[align=left, right,yshift=0.2em, xshift=-0.1em] {$z$};

			\draw (5.95,-0.95) circle [x radius=0.5cm, y radius=0.63cm, rotate=45];
			\draw (5.9,-0.95) node{$-A$};

			\draw (7.1,-2.1) node { .};
			\draw (7.2,-2.2) node { .};
			\draw (7.3,-2.3) node { .};
			
			\draw (5.6,0.6) node { .};
			\draw (5.7,0.7) node { .};
			\draw (5.8,0.8) node { .};

			\draw (6,-2) node{$\Gamma'$};
			\draw (6,-2.7) node{(b)};

			\draw (3.5,-3) node{(12)};
			
		\end{tikzpicture}
		
		\caption{Construction of $\Gamma'$ from $\Gamma$. Realization of $\Gamma'$ by a simple Morse function implies realization of $\Gamma$, which leads to combinatorial modification number (12).}\label{figure:removing_two_extrema_case_b}
	\end{figure}
	
	Let $q\colon M \to \reeb{f'}$ be the quotient map, $\overline{f'}\colon \reeb{f'}\to \R$ be the induced function and let $a=\overline{f'}(x)$. Let $p$ be a point between $y'$ and $z$ in $\reeb{f'}$. By $[z',x]$, $[y',p]$ and $[p,z]$ we denote the segments joining appropriate points in $\reeb{f'}$. We will construct a simple Morse function $f$ on $M$ which realizes $\Gamma$. First, take an orientation-preserving diffeomorphism $h$ of $[\overline{f'}(y'),\overline{f'}(p)]$ onto $[\overline{f'}(y'),a+\varepsilon]$. Next, let $W=q^{-1}([p,z])$ and $W'=q^{-1}([z',x])$. Take ordered and simple Morse functions 
	$$g\colon W \to [a+\varepsilon,a+2\varepsilon]  \text{ on the triad } \left(W,q^{-1}(p) \sqcup q^{-1}(z),\varnothing\right)\!,$$
	$$g'\colon W' \to [a-\varepsilon,a]  \text{ on the triad } \left(W',\varnothing,q^{-1}(z') \sqcup q^{-1}(x)\right)\!,$$ with exactly one critical point being extremum (maximum and minimum, respectively). By~Propositions \ref{proposition:correspondece_degree_index} and \ref{proposition:ordered_has_tree} the Reeb graphs $\reeb{g}$ and $\reeb{g'}$ are homeomorphic to a~small neighbourhoods of vertices $w$ and $w'$ in $\Gamma$, respectively. Take a~submanifold $Q = q^{-1}(-A)$ with boundary $\partial Q = q^{-1}(z) \sqcup q^{-1}(z')$ and let $h_Q$ be an orientation-reversing diffeomorphism of the interval $f'(Q) = \overline{f'}(-A) = [\overline{f'}(z),\overline{f'}(z')]$ onto $[a,a+\varepsilon]$. Now, we define a~Morse function $f$ on $M$ by
	\begin{eqnarray*}
		f (s) =  \begin{cases}
			h(f'(s))	        &\text{\ \ \ \ if $s \in q^{-1}([y',p])$,}\\
			g(s)	        &\text{\ \ \ \ if $s \in W\!$,}\\
			g'(s)	        &\text{\ \ \ \ if $s \in W'$,}\\
			h_Q(f'(s))	      &\text{\ \ \ \ if $s \in Q$,}\\
			f'(s)	        &\text{\ \ \ \ in other cases.}
		\end{cases}
	\end{eqnarray*}
	
	It is easily seen that $f$ realizes $\Gamma$ and can be changed to be simple since connected components of level sets contains at most one critical point. 
\end{proof}

\begin{remark}
	In fact, in the above proof we have shown that for any graph $\Gamma$ with good orientation there is a finite sequence of combinatorial modifications (1) -- (12) transforming the initial graph to $\Gamma$ up to vertices of degree $2$.
\end{remark}	

\section{Further directions}

In this paper we have resolved the realization problem for Reeb graphs \mbox{(Problem~\ref{main_problem})} up to orientation-preserving homeomorphism of graphs. The next natural step is to improve this result constructing a function that would realize a given graph as the Reeb graph up to isomorphism. As it can probably be done using smooth functions with also degenerate critical points, the interesting question is about a realizability up to isomorphism by a Morse function. From the point of view of \cite[Theorem 5.6]{Michalak} it should depend on some conditions on the number of vertices of degree $2$ and, for example, the homological structure of a~manifold. 


The above problems can also be considered in the case of compact manifolds with boundary.

We would also ask for a description of the class of continuous functions $f\colon M \to \R$ on a manifold $M$ for which the quotient space $\reeb{f}$ is a finite graph (cf. Remark \ref{remark:reeb_nubmer_continuous}). In fact, to obtain the inequality $\beta_1(\reeb{f}) \leq \corank(\pi_1(M))$ it suffices to require that $\reeb{f}$ is semilocally simply connected and $\pi_1(\reeb{f})$ is free (cf. \cite{Gelbukh:filomat}).

The last question which we would like to point out is how to find a function on a given manifold $M$ whose Reeb graph has the maximum possible cycle rank equal to $\reeb{M}$. Proposition \ref{proposition:ordered_has_tree} shows that one can easily get a function with zero cycles in the Reeb graph. Also, the combinatorial modifications of Reeb graphs, especially the modification $(7)$, allow us to decrease the cycle rank. We are looking for a reverse method of increasing the cycle rank. Our general aim is to provide conditions on functions to ensure the maximum cycle rank of Reeb graph. For example, if $M\subset \R^N$ is embedded in the Euclidean space, then for any point $z\in \R^N$ consider a function $f_z \colon M \to \R$ given by $f_z(x) = ||x-z||^2$. It is known that it is a Morse function for almost all $z\in \R^N$. We will try to answer the following question: what is the distribution of cycle ranks among Reeb graphs of functions $f_z$? 


\address
	
\end{document}